\documentclass[reqno,twoside,11pt]{amsart}
\usepackage{amssymb,amsfonts,amsthm,amsmath,mathrsfs}
\usepackage{graphicx}
\usepackage{cite,xcolor}
\usepackage{multicol}
\usepackage{multirow}

\newcommand{\beq}{\begin{equation}}
\newcommand{\eeq}{\end{equation}}
\newcommand{\beqs}{\begin{equation*}}
\newcommand{\eeqs}{\end{equation*}}
\newcommand{\ba}{\begin{array}}
\newcommand{\ea}{\end{array}}
\newcommand{\beas}{\begin{eqnarray*}}
\newcommand{\eeas}{\end{eqnarray*}}
\newcommand{\bea}{\begin{eqnarray}}
\newcommand{\eea}{\end{eqnarray}}
\newcommand{\bal}{\begin{align}}
\newcommand{\eal}{\end{align}}

\newcommand{\bals}{\begin{align*}}
\newcommand{\eals}{\end{align*}}

\newcommand{\R}{\ensuremath{\mathbb R}}

\newcommand{\N}{\ensuremath{\mathbb N}}

\newcommand{\bds}{\begin{displaystyle}}
\newcommand{\eds}{\end{displaystyle}}

\renewcommand{\eqref}[1]{(\ref{#1})}

\def\longequals{\mathbin{=\kern-2pt=}}
\def\eqdef{\mathbin{\buildrel \rm def \over \longequals}}

\def\varep{\varepsilon}

\newcommand{\dx}[2]{ \frac{\partial  {#1} }   {\partial {#2}} }
\newcommand{\ddx}[3]{ \frac{\partial^2  {#1} } { \partial {#2} \partial {#3}  } } 

\newcommand{\remove}[1]{} 
\renewcommand{\remove}[1]{#1} 

\newtheorem{theorem}{Theorem}[section]

\newtheorem{lemma}[theorem]{Lemma}
\newtheorem{corollary}[theorem]{Corollary}
\newtheorem{proposition}[theorem]{Proposition}

\newtheorem{remark}[theorem]{\bf{Remark}}

\theoremstyle{remark}
\newtheorem{example}[theorem]{\bf{Example}}

\numberwithin{equation}{section}

\definecolor{darkred}{rgb}{.70,.12,.20}

\definecolor{darkgreen}{rgb}{.20,.52,.14}

\def\myclearpage{}

\newcommand{\U}{\mathbf{u} }
\newcommand{\V}{\mathbf{v} } 
\newcommand{\G}{\mathbf{G} }
\newcommand{\x}{\mathbf x}
\newcommand{\Am}{ \underline{\mathbf{A} } }
\newcommand{\Bm}{ \underline{\mathbf{B} } }

\newcommand{\Idn}{ \underline{\mathbf{I}}_n }
\newcommand{\Rmax}{ R_{\max} }

\newcommand{\fo}{f_0}

\newcommand{\deltaz}{N}
\newcommand{\deltazero}{\Delta_9}
\newcommand{\deltao}{\Delta_1}  
\newcommand{\deltatw}{\Delta_2}  
\newcommand{\deltath}{\Delta_4}  
\newcommand{\deltathp}{\Delta_5}  
\newcommand{\deltafo}{\Delta_3}  

\newcommand{\deltafi}{\Delta_6}
\newcommand{\deltasi}{\Delta_7}
\newcommand{\deltafiT}{\Delta_{6,T}}
\newcommand{\deltasiT}{\Delta_{7,T}}

\newcommand{\deltaei}{\Delta_{11}}  
\newcommand{\deltani}{\Delta_{13}}

\newcommand{\deltaten}{\Delta_{8}}    
\newcommand{\deltaele}{\Delta_{10}}    
\newcommand{\deltatwe}{\Delta_{12}}    

\newcommand{\mscrD}{\mathscr D}
\newcommand{\mscrU}{\mathscr U}
\newcommand{\mscrO}{\mathcal O}

\newcommand{\kap}{\kappa_0}
\newcommand{\kapo}{\kappa_1}
\newcommand{\kaptw}{\kappa_2}
\newcommand{\Sph}{\mathcal S}

\newcommand{\condEa}{\textbf{(E1)}}
\newcommand{\condEb}{\textbf{(E2)}}

\title{A family of steady two-phase generalized Forchheimer flows and their linear stability analysis}

\author{Luan T. Hoang, Akif Ibragimov and Thinh T. Kieu$^\dag$}
\address{Department of Mathematics and Statistics, Texas Tech University, Box 41042
Lubbock, TX 79409--1042, U.S.A.}
\email{luan.hoang@ttu.edu}
\email{akif.ibraguimov@ttu.edu}
\email{thinh.kieu@ttu.edu}
\address{$^\dag$ Corresponding author}
\date{\today}

\begin{document}

\begin{abstract}
We model multi-dimensional two-phase flows of incompressible fluids in porous media using generalized Forchheimer equations and the capillary pressure.
Firstly, we find a family  of steady state solutions whose saturation and pressure are radially symmetric and velocities are rotation-invariant. 
Their properties are investigated based on relations between the capillary pressure, each phase's relative permeability and Forchheimer polynomial.
Secondly, we analyze the linear stability of those steady states. 
 The linearized system is derived and reduced to a parabolic equation for the saturation. This equation has a special structure  depending on the steady states which we exploit to prove two new forms of the lemma of growth of Landis-type in both bounded and unbounded  domains. Using these lemmas, qualitative properties of the solution of the linearized equation are studied in details. In bounded domains, we show that the solution decays exponentially in time. In unbounded domains, in addition to their stability, the solution decays to zero as the spatial variables tend to infinity. The Bernstein technique is also used in estimating the velocities. All results have a clear physical interpretation.
\end{abstract}

\maketitle

\begin{center}
 { \large \textit{Dedicated to the Memory of Evgenii Mikhailovich Landis (1921--1997)} }
\end{center}

\tableofcontents

\pagestyle{myheadings}\markboth{L. Hoang, A. Ibragimov and T. Kieu}
{Steady Two-phase Generalized Forchheimer Flows and Their Linear Stability Analysis}
\myclearpage
\section{Introduction}\label{Intro}

In this paper, we study two-phase flows of incompressible fluids in porous media with each phase subjected to a Forchheimer equation.
Forchheimer equations are often used by engineers to take into account the deviation from Darcy's law in case of high velocity, see e.g. \cite{BearBook,Straughanbook}. The standard Forchheimer equations are two-term law with quadratic nonlinearity, three-term law with cubic nonlinearity, and power law with a non-integer power less than two (see again \cite{BearBook,Straughanbook}).
These models are extended to the generalized Forchheimer equation of the form 
\beq 
g(|\mathbf u|)\mathbf u=-\nabla p,
\eeq
where $\mathbf u(\x,t)$ is the velocity field, $p(\x,t)$ is the pressure, and $g(s)$ is a generalized polynomial of arbitrary order (integer or non-integer) with positive coefficients.
This equation was intensively analyzed for single-phase flows from mathematical and applied point of view  in \cite{ABHI1,HI1,HI2,HIKS1,HKP1}. 
Its study for two-phase flows  was later initiated in \cite{HIK1}. 
Regarding two-phase flows in porous media, it is always a challenging subject even for Darcy's law. 
Their models involve a complicated system of nonlinear partial differential equations (PDE) for pressures, velocities, densities and saturations with many parameters such as porosity, relative permeability functions and capillary pressure function.
Current analysis of two-phase Darcy flows in literature is mainly focused on the existence of weak solutions \cite{Cances2009a,Cances2009b,Brull2008} and their regularity \cite{Kruzhkov77, Kruzhkov85, AltDi84,AltDi85,DiIUMJ2011}. 
However, questions about the stability and dynamics are not answered.
The nonlinearity of the relative permeabilities and capillary pressure and their imprecise characteristics near the extreme values make it hard to analyze the modeling PDE system.
The two-phase generalized Forchheimer flows are even more difficult due to the additional nonlinearity in the momentum equation.
For example, unlike the Darcy flows, there is no Kruzkov-Sukorjanski transformation \cite{Kruzhkov77} to convert the system to a convenient form for the total velocity. Therefore, new methods are needed for the Forchheimer flows. 
In \cite{HIK1}, we study the one-dimensional case using a novel approach. 
We will develop the techniques in \cite{HIK1} further to investigate the multi-dimensional case in this article.


We consider $n$-dimensional two-phase flows in porous media with constant porosity  $\phi$ between $0$ and $1$.  
Here the dimension $n$ is greater or equal to $2$, even though in practice we only need $n=2,3$.
Each position $\x=(x_1,x_2,\ldots,x_n)\in\R^n$ in the medium is considered to be occupied by two fluids called phase 1 (for example, water) and phase 2 (for example, oil).

Saturation, density, velocity, and pressure for each $i$th-phase ($i=1,2$) are $S_i\in[0,1]$, $\rho_i\ge 0$, $\bf u_i\in \R^n$ and $p_i\in\R$, respectively. The saturation functions naturally satisfy 
\beq\label{Snat}
  S_1+S_2=1.
\eeq

Each phase's velocity is assumed to obey the generalized Forchheimer equation:
\beq\label{Forch} 
g_i(|\mathbf u_i|)\mathbf u_i=-\tilde f_i(S_i)\nabla p_i,\quad i=1,2,\eeq 
where $\tilde f_i(S_i)$ is the relative permeability for the $i$th phase, and $g_i$ is of the form
\beq\label{gform}
g_i(s)=a_0 s^{\alpha_0} +a_1 s^{\alpha_1}+\ldots +a_N s^{\alpha_N}, \quad s\ge 0,
\eeq 
with $N\ge 0$, $a_0>0,$ $a_1,\ldots a_N\ge 0$, $\alpha_0=0<\alpha_1<\ldots \alpha_N$, all $\alpha_1,\ldots\alpha_N$ are real numbers.
The above $N$, $a_j$, $\alpha_j$ in \eqref{gform} depend on each $i$.
We call $g_i(s)$ in \eqref{gform} the Forchheimer polynomial of \eqref{Forch}.

Conservation of mass commonly holds for each of the phases:
\beq\label{mass0}
  \partial_t(\phi \rho_i S_i)+ {\rm div} (\rho_i \mathbf u_i)=0, \quad i=1,2.\eeq

Due to incompressibility of the phases, i.e. $\rho_i=const.>0$, Eq. \eqref{mass0} is reduced to
\beq\label{mass}
\phi \partial_t S_i + {\rm div}\, \mathbf u_i=0,\quad i=1,2.\eeq

Let $p_c$ be the capillary pressure between two phases, more specifically,
\beq\label{cap0} p_1-p_2=p_c.\eeq
 
Hereafterward, we denote $S=S_1$. The relative permeabilities and capillary pressure are re-denoted as functions of $S$, that is, $\tilde f_1(S_1)=f_1(S)$, $\tilde f_2(S_2)=f_2(S)$ and $p_c=p_c(S)$. Then \eqref{Forch} and \eqref{cap0} become
\beq\label{Forch-S} 
g_i(|{\mathbf u_i}|){\mathbf u_i}=- f_i(S)\nabla p_i,\quad i=1,2,\eeq 
\beq\label{cap} p_1-p_2=p_c(S).\eeq

By scaling time, we can mathematically consider, without loss of generality, $\phi=1$.

By \eqref{Snat} and \eqref{mass}: 
\beq\label{Sdot} S_t=- {\rm div}\, {\bf u}_1,\quad S_t={\rm div}\, {\bf u}_2.\eeq

For $i=1,2$, define the function $\G_i({\bf u})=g_i(|{\bf u}|){\bf u}$ for ${\bf u}\in \R^n$. 
Then by \eqref{Forch-S},
\beq\label{pprime}
\G_i({\bf u}_i)=-f_i(S)\nabla p_i, \quad \text{or,}\quad \nabla p_i=- \frac{\G_i({\bf u}_i)}{f_i(S)}.
\eeq
Taking gradient of the equation \eqref{cap} we have
\beq\label{pcprime} \nabla p_1- \nabla p_2=p_c'(S) \nabla S.
\eeq
Substituting \eqref{pprime} into \eqref{pcprime} yields
\beqs
\frac{g_2(|\mathbf u_2|)\mathbf u_2}{f_2(S)} - \frac{g_1(|\mathbf u_1|)\mathbf u_1}{f_1(S)}=p_c'(S)\nabla S,
\eeqs
hence
\beqs  
F_2(S) g_2(|{\bf u}_2|)\mathbf u_2-F_1(S) g_1(|\mathbf u_1|)\mathbf u_1=\nabla S , \eeqs
where
\beq \label{Fdef}
F_i(S)=\frac{1}{p_c'(S)f_i(S)},\quad i=1,2.
\eeq 

 In summary we study the following PDE system for $\x\in\R^n$ and $t\in \R$:
\begin{subequations}\label{mainsys}
\begin{align}
&\label{S01} 0\le S=S(\mathbf x,t)\le 1,\\
&\label{Sdot1} S_t=- {\rm div}\,  {\mathbf u}_1 ,\\
&\label{Sdot2} S_t={\rm div}\,  {\mathbf u}_2,\\
&\label{Sgstar} \nabla S=F_2(S)\G_2(\U_2)- F_1(S)\G_1(\U_1).
\end{align}
\end{subequations}


This paper is devoted to studying system \eqref{mainsys}.
We will obtain a family of non-constant steady states with particular geometric properties.
Specifically, the saturation and pressure are functions of $|\x|$, while each phase's velocity is $\x$ multiplied by a radial scalar function.
Their properties, particularly, the behavior as $|\x|\to\infty$, will be obtained.
For the stability study, we linearize system \eqref{mainsys} at these steady states. We deduce from this linearized system a parabolic equation for the saturation.
In bounded domains, we establish the lemma of growth in time and prove the exponential decay of its solutions in sup-norm as time $t\to\infty$.
In unbounded domains, we prove the maximum principle and the stability. 
Furthermore, we show that the solutions go to zero as the spatial variables tend to infinity.

The paper is organized as follows.
In section \ref{Steady} we find the family of non-constant steady states described above. 
Various sufficient conditions are given for their existence in unbounded domains (Theorems \ref{globalODE}).
Their asymptotic behavior as $|\x|\to\infty$ is studied in details.
In section \ref{linearization}, we linearize the originally system at the obtained steady states.
We derive a parabolic equation for the saturation which will become the focus of our study.
It is then converted to a convenient form for the study of sup-norm of solutions. Such a conversion is possible thanks to the special structure of the equation and of the steady states. Preliminary properties of the coefficient functions of this linearized equation are presented.
Section \ref{bounded} is focused on the study of the linearized equation for saturation in bounded domains.
We prove the asymptotic stability results (Theorems \ref{Bsig} and \ref{mainv}) by utilizing a variation of Landis's lemma of growth in time variable (Lemma \ref{grow}).  The Bernstein's a priori estimate technique is used in proving interior continuous dependence of the velocities on the initial and boundary data (Proposition \ref{gradw}).
In section \ref{unbounded}, we study the linearized equation in an (unbounded) outer domain. 
The maximum principle (Theorem \ref{maximumprinciple}) is proved and used to obtain the stability of the zero solution (Theorems \ref{uniqS3} and \ref{sta1}, part (ii)).
We also prove a lemma of growth in the spatial variables (Lemma \ref{Cmp}) 
by constructing particular barriers (super-solutions) using the specific structure of the linearized equation  for saturation (Lemma \ref{Sub-sol}). 
Using this, we prove a dichotomy theorem on the solution's behavior (Lemma \ref{cyls}), and ultimately show that the solution, on any finite time interval, decays to zero as $|\x|\to\infty$. 
For time tending to infinity, we find an increasing, continuous function $r(t)>0$  with $r(t)\to\infty$ as $t\to\infty$ such that along any curve $\x(t)$ with $|\x(t)|\ge r(t)$, the solution goes to zero. (See Theorems \ref{uniqS3} and \ref{sta1}, part (iii).)
It is worth mentioning that the asymptotic stability in sup-norm in section \ref{bounded} and behavior of the solution at spatial infinity have their own merits in the qualitative theory of linear parabolic equations.

%

\myclearpage
\section{Special steady states}\label{Steady}

In this section we find and study steady states which processes some symmetry. 

Assume $p_i$ and $S$ are radial functions. We can write
\beq 
p_i({\bf x},t) =p_i(r,t),\quad S({\bf x},t)=S(r,t), \quad\text{where } r=|\x|=\big(\sum_{i=1}^n x_i^2\big)^{1/2}.
\eeq  
Denote ${\bf e}_r=\x/|\x|$. By \eqref{Forch-S},
\beq\label{ui0}
g_i(|\U_i|){\U_i} =- f_i(S)\frac {\partial p_i}{\partial r}\cdot \frac {\bf x} {r}=- f_i(S)\frac {\partial p_i}{\partial r}{\bf e}_r.
\eeq
Noting in \eqref{ui0} that $f_i(S)\frac {\partial p_i}{\partial r}$ is radial, then clearly $|{\bf u}_i|$ is also radial and we have
\beq
{\bf u}_i=u_{ir} {\bf e}_r,\quad \text{where }u_{ir} ={\bf u}_i\cdot {\bf e}_r=u_{ir}(r,t).
\eeq
Therefore  
\beq
{\rm div}\, \U_i =\frac 1{r^{n-1}}\frac {\partial}{\partial r} (r^{n-1}u_{ir} )
\eeq
and, from \eqref{Sgstar},
\beq\label{Feq0}
F_2(S)g_2(|\U_2|)\U_2-F_1(S)g_1(|\U_1|){\U_1}=\nabla S = \frac{\partial S}{\partial r}{\bf e}_r.
\eeq
Taking the scalar product of both sides of \eqref{Feq0} with ${\bf e}_r$ we obtain   
\beq\label{Feq}
G_2(u_{2r})F_2(S)-G_1(u_{1r})F_1(S)= \frac{\partial S}{\partial r},
\eeq
where 
\beq\label{Gidef}
G_i(u)=g_i(|u|)u\quad\text{for }u\in \R.
\eeq

We will study $S(r,t)$ and $u_i(r,t)\eqdef u_{ir}$ ($i=1,2$) as functions of independent variables $(r,t)\in (0,\infty)\times \R$.
The system \eqref{mainsys} becomes
\begin{subequations}\label{Rsys}
\begin{align}
&\label{S} 0\le S\le 1,\\
&\label{St1} \frac{\partial S}{\partial t}=-r^{1-n}\frac{\partial }{\partial r} (r^{n-1}u_1) ,\\
&\label{St2} \frac{\partial S}{\partial t}=r^{1-n}\frac{\partial }{\partial r} (r^{n-1}u_2),\\
&\label{Sx} \frac{\partial S}{\partial r}=G_2(u_2)F_2(S)- G_1(u_1)F_1(S).
\end{align}
\end{subequations}

We make basic assumptions on the relative permeabilities and capillary pressure.
 
\textbf{Assumption A.}
\begin{subequations}
\beq f_1,f_2\in C([0,1])\cap C^1((0,1)),\eeq 
\beq\label {f1f2}   f_1(0)=0,\quad f_2(1)=0,\eeq
\beq f_1'(S)>0,\quad f_2'(S)<0 \text{ on } (0,1).\eeq 
\end{subequations}

\textbf{Assumption B.}
\beq\label{pcincr} p_c'\in C^1((0,1)),\quad p'_c(S)>0 \text{ on } (0,1).\eeq

We find steady state solutions $(S,u_1,u_2)=(S(r),u_1(r),u_2(r))$ for system \eqref{Rsys} in the domain $[r_0,\infty)$ for a fixed $r_0>0$.

From \eqref{St1}, we have $\frac{d}{dr} (r^{n-1}u_i)=0$, hence 
\beq\label{ui}
u_i(r)=c_i r^{1-n}, \quad \text{where } c_i=const.,\quad i=1,2.
\eeq
Substituting \eqref{ui} into \eqref{Sgstar} yields  
\beq\label{Sprime}
S'=G_2(c_2r^{1-n} )F_2(S)-G_1(c_1r^{1-n})F_1(S)\quad \text{for } r>r_0.  
\eeq
The rest of this section is devoted to studying the following initial value problem with constraints:   
\beq\label{IVP}
S'= F(r,S(r))\quad \text{for } r>r_0, \quad S(r_0)=s_0, \quad  0<S(r)<1. 
\eeq
where $s_0$ is always a number in $(0,1)$ and
\beqs
F(r, S(r)) =G_2(c_2r^{1-n} )F_2(S)-G_1(c_1r^{1-n})F_1(S). 
\eeqs

First we state a standard local existence theorem.
\begin{theorem}\label{localODE} 
There exist a maximal interval of existence $[r_0,R_{\rm max})$, where $R_{\rm max}\in (r_0,\infty]$, and a unique solution $S\in C^1([r_0,R_{\rm max}))$ of \eqref{IVP} on $(r_0,R_{\rm max})$.
 Moreover, if $R_{\rm max}$ is finite then either
 \beq\label{odelim}
 \lim_{r\to R_{\rm max}^-} S(r)=0 \quad \text{ or}   \lim_{r\to R_{\rm max}^-} S(r)=1.
 \eeq
\end{theorem}
\begin{proof} 

Under Assumption B,  $F(r,S)$ is continuous and locally Lipschitz for the second variable for all $r\in(r_0, \infty) $, $S\in (0,1)$. 
The existence of the unique solution $S\in C^1([r_0,R_{\rm max});(0,1))$ on the maximal interval $[0,R_{\rm max})$ is classical. 

Assume $R_{\max}<\infty$. For given $0<\varep \le \varep_0\eqdef \min\{1/4,R_{\rm max}/2\}$, let $K=[r_0,R_{\max} ]\times[\varep, 1-\varep]$.  We claim that there is $R_\varep \ge r_0$ such that $ (r,S(r))\notin K$ for all $ r\in (R_\varep, R_{\max})$. Suppose not, then there is the sequence $r_i\to R_{\max}$ as $i\to\infty$ such that $( r_i, S(r_i))\in K$ for all $i$. 
Choose $N>0$ such that for all $i\ge N$, 
  \beqs
    \{ (r,S):   |r-r_i|\le \varep/2 \text { and } |S-S(r_i)|\le \varep/2\} \subset K'  ,
  \eeqs
  where $K'=[r_0,R_{\max}+\varep/2 ]\times[\varep/2, 1-\varep/2]$. According to  the local Existence and Uniqueness theorem (Theorem 3.1 p. 18 in\cite{JHale78}) the solution starting at point $ (r_i, S(r_i))$ exists on the interval $[r_i,r_i+d)$, where $d=\min\{\frac 1 L, \frac \varep 2,\frac{\varep} {2M} \}$ with $M=\max_{K'} |F(r,S)|$ and $L$ being the Lipschitz constant for $F$ in $K'$.
Note that $d$ is independent of $i$. 
Let $i$ be sufficiently large such that $r_i +d > R_{\max}$, then solution $S(r)$ exists beyond $R_{\max}$ which is a contradiction to maximality of $R_{\max}$.  Hence our claim is true. Now using the continuity of $S(r)$ we have 
\beq\label{choice} 
\text{either  } S(r)> 1-\varep, \forall r\in(R_{\varep}, R_{\max}) \text{  or } S(r)< \varep, \forall r\in(R_{\varep}, R_{\max}).
\eeq      
In particular, for $\varep=\varep_0$  we have either (a) $S(r)>1-\varep_0, \forall r\in (R_{\varep_0}, R_{\max})$, or (b) $S(r)<\varep_0, \forall r\in(R_{\varep_0}, R_{\max})$. 
In case (a), it is easy to see from \eqref{choice} that for $0<\varep<\varep_0$, $S(r)> 1-\varep, \forall r\in(R'_{\varep}, R_{\max})$
where $R'_\varep=\max\{R_{\varep_0}, R_\varep\}$. Thus, $\lim_{r\to R_{\max}^-} S(r) =1$. Similarly, for the case (b) we have  $\lim_{r\to R_{\max}^-} S(r) =0$. The  proof is complete.
\end{proof}

Next, we are interested in the case  $R_{\rm max}=\infty$. First, we find sufficient conditions for that. 
We need to make the following assumptions on the relative permeabilities and capillary pressure:
\beq\label{pcfcond} \lim_{S\to 0} p'_c(S)f_1(S)=\lim_{S\to 1} p'_c(S) f_2(S)=+\infty. \eeq 
These are our interpretation of experimental data (c.f. \cite{BearBook}), especially of those obtained in \cite{BrooksCorey64}. They cover certain scenarios of two-phase fluids in reality.

By \eqref{Fdef} and \eqref{pcfcond}, $F_1$ and $F_2$ can now be extended to functions of class $C([0,1])\cap C^1((0,1))$ and satisfy 
\beq\label{zeroends}
F_1(0)=F_1(1)=F_2(0)=F_2(1)=0.
\eeq
Therefore the right hand side of \eqref{Sgstar} is well-defined for all $S\in [0, 1]$.
Note that
\beq\label{quotlim}
\lim_{S\to 0^+} \frac{F_1(S)}{F_2(S)}  =\lim_{S\to 1^-} \frac{F_2(S)}{F_1(S)} =\infty.
\eeq

The following additional conditions on $F_1$ and $F_2$ will be referred to in our considerations:
 \beq \label{Sto0}  
\limsup_{S\to 0^+} F_1'(S)<\infty,
\eeq
 \beq \label{Sto1F1} 
\liminf_{S\to 1^-}F'_1(S)>-\infty.
 \eeq
 \beq \label{Sto1} 
\liminf_{S\to 1^-}F'_2(S)>-\infty.
 \eeq
 \beq \label{Sto0F2}  
\limsup_{S\to 0^+} F_2'(S)<\infty,
\eeq


\begin{theorem}\label{globalODE} Assume \eqref{pcfcond} and $c_1^2+c_2^2>0$. Then $R_{\max} $ in Theorem \ref{localODE} is infinity, that is, the solution $S(r)$ of \eqref{IVP} exists on $[r_0,\infty)$,  in the following cases

{Case 1a.} $c_2\le 0<c_1$  and \eqref{Sto0}.
\quad\quad\quad\quad
{Case 1b.} $c_1=0> c_2$  and \eqref{Sto0F2}.

{Case 2a.} $c_1\le 0<c_2$ and \eqref{Sto1}.
\quad\quad\quad\quad
{Case 2b.} $c_2=0> c_1$  and \eqref{Sto1F1}.

{Case 3.} $c_1,c_2>0$ and \eqref{Sto0}, \eqref{Sto1}. 

{Case 4.} $c_1,c_2<0$.
\end{theorem}
\begin{proof} 
Suppose $R_{\max}<\infty$. We consider the following four  cases.  
 
 {\bf Case 1.} $c_2\le 0\le c_1$. We provide the proof of  Case 1a, while Case 1b can be proved similarly. 
We have $F(r,S)<0$ for all $r\in [r_0, \Rmax)$. Thus $S'<0$ for all $r\in [r_0, \Rmax)$. By Theorem~\ref{localODE},  
\beq\label{limS} \lim_{r\to R_{\rm max}^-} S(r)=0.\eeq 
Note that $G_1(c_1r^{1-n})$ and $G_2(c_2r^{1-n})$ are bounded,  and $G_1(c_1r^{1-n})$ is  bounded below by a positive number  on $[r_0,R_{\rm max}]$.
Combining these facts with relation \eqref{quotlim}, we infer that there are $\delta>0$ and $C_1,C_2>0$ 
such that for $r\in [0,R_{\rm max})$ and $S\in (0,\delta)$,
 \beq\label{eqv0}
-C_1F_1(S)\le  F(r,S) \le -C_2 F_1(S).
\eeq 
By \eqref{limS}, there is $r_1\in (0,R_{\rm max})$ such that $S(r)<\delta$ for all $r\in [r_1,R_{\rm \max})$. 
Define $Y(r)=F_1(S(r))$. By \eqref{Sto0}, there are $\tilde r\in (r_1,\Rmax)$  and $C_3>0$
\beq\label{hprime}
 F'_1(S(r))<C_3 \text{ for all } r\in (\tilde r,R_{\rm max}).
\eeq 
For $r\in (\tilde r,\Rmax)$,  using \eqref{eqv0} we have
\beq\label{Yeq1} 
Y'(r) = F'_1(S)S'= F'_1(S)F(r,S)\ge-C F'_1 (S)  F_1(S)>-C_4 F_1(S) = - C_4 Y(r),
\eeq
where $C>0$, $C_4=C C_3>0.$ Thus \eqref{Yeq1} gives 
\beq\label{Yeq}  Y(r)\ge Y(\tilde r) e^{-C_4(r-\tilde r)},\quad r \in [\tilde r,\Rmax).\eeq
We have from \eqref{limS} and \eqref{zeroends} that
\beq\label{limY}
\lim_{r\to \Rmax^-} Y (r)=0.
\eeq 
Let $r\to \Rmax^-$ in \eqref{Yeq} and using \eqref{limY}, we obtain $0\ge Y(\tilde r)e^{-C_4 (\Rmax-\tilde r)}>0$ which is a contradiction. 
       
 {\bf Case 2.} $c_1\le 0\le c_2$. Both Case 2a and  2b are proved similarly. Consider  Case 2a. Since $F(r,S)>0$ for all $r\in [r_0, \Rmax)$,  $S'>0$ for all $r\in [r_0, \Rmax)$ therefore by Theorem~\ref{localODE}, $\lim_{r\to R_{\rm max}^-} S(r)=1$. 
 
Let $X=1-S$. Then
  $  \lim_{r\to R_{\rm max}^-} X(r)=0$    and    
 \beq\label{Xeq}
  X' =  -S'= - F(r, 1-X)=\tilde F(r, X)=G_1(c_1r^{1-n})\tilde F_1(X)- G_2(c_2r^{1-n})\tilde F_2(X), 
 \eeq
 where $ \tilde F_i(X)= F_i(1-X)$. 
Similar to the proof of Case 1a, there are $\delta>0$ and $C_1,C_2>0$ such that    
 \beq\label{equiv1}
-C_1 \tilde F_2(X)\le  \tilde F(r,X) \le -C_2 \tilde F_2(X),
\eeq 
for all $r\in [r_0,\Rmax]$ and $X\in(0,\delta)$.
Note that condition \eqref{Sto1} is equivalent to $\limsup_{X\to 0^+} \tilde F_2'(X)<\infty $. 
Repeating the proof in  Case 1a with $\tilde F_2$ instead of $ F_1$ leads  to a contradiction. 
  
  {\bf Case 3.}  According to Theorem \ref{localODE} we have two cases.

    (i) Case $\lim_{r\to R_{\rm max}^-} S(r)=0$. 
By \eqref{quotlim} there are constants $C_1,C_2>0$ and $\delta>0$  such that 
\beqs
-C_1F_1(S)\le  F(r,S) \le -C_2 F_1(S).
\eeqs 
for all $r\in [r_0,R_{\rm max}]$ and $S\in (0,\delta)$.
Also, there is $r_1\in(0,\Rmax)$ such that $S(r)<\delta$ for all $r\in(r_1,R_{\rm max})$.
Then the exact argument for Case 1a yields a contradiction.
         
    (ii) Case $\lim_{r\to R_{\rm max}^-} S(r)=1$. By \eqref{quotlim}, there $\delta>0$ and $C_1,C_2>0$ such that
     \beqs 
      C_1F_2(S)\le  F(r,S) \le C_2 F_2(S)
   \eeqs 
for all $r\in[r_0,R_{\rm max}]$ and $S\in (1-\delta,1)$.
Then the proof is proceeded similar to  Case 2a under condition \eqref{Sto1} to obtain a contradiction.
     
  {\bf Case 4.} Again, according to Theorem \ref{localODE} we have two cases.

(i) Case $\lim_{r\to R_{\rm max}^-} S(r)=0$. 
By \eqref{quotlim}, there are $\delta>0$ and $C_1,C_2>0$ such that
     \beqs 
       0< C_1F_1(S)\le  F(r,S) \le C_2 F_1(S)
   \eeqs 
for all $r\in[r_0,R_{\rm max}]$ and $S\in (0,\delta)$.
Let $r_1$ be as in {\bf Case 3}(i). Then for $r\in (r_1,R_{\rm max})$ we have $S'(r)>0$, and hence $S(r)\ge S(r_1)>0$ which contradicts the fact $\lim_{r\to R_{\rm max}^-} S(r)=0$.

(ii) Case $\lim_{r\to R_{\rm max}^-} S(r)= 1$. 
By \eqref{quotlim}, there are $\delta>0$ and $C_1,C_2>0$ such that
\beqs 
 -C_1F_2(S)\le F(r,S)\le -C_2F_2(S)<0
\eeqs
for all $r\in[r_0,R_{\rm max}]$ and $S\in (1-\delta,1)$.
There is $r_1\in(r_0,\Rmax)$ such that $S(r)\in (1-\delta,1)$ for all $r\in(r_1,\Rmax)$.
Thus $S'(r)<0$ for all $r\in(r_1, \Rmax)$ which gives $S(r)\le S(r_1)$. Letting $r\to\Rmax$ yields $1\le S(r_1)<1$. This is a contradiction.  

From all the above contradictions, we must have $\Rmax=\infty$ and the proof is complete.   
\end{proof}

To study $S(r)$ as $r\to\infty$, for the solution $S(r)$ in the Theorem \ref{globalODE} we will need   function $h(r)\in (0,1)$ such that
\beq\label{RHS}
 G_2(c_2r^{1-n})F_2(h(r))-G_1(c_1r^{1-n})F_1(h(r))=0.
 \eeq
To prove existence of such function consider $c_1c_2\ne 0$. Then \eqref{RHS} is equivalent to 
\beqs
 \frac{f_1(h(r))}{f_2(h(r))}=\frac{c_1 g_1(|c_1|r^{1-n}) }{c_2g_2(|c_2|r^{1-n})}.
\eeqs 
Since $f\eqdef f_1/f_2$ is strictly increasing and maps $(0,1)$ onto $(0,\infty)$, we can solve 
\beq \label{equil}
h(r)=f^{-1}\Big( \frac{c_1 g_1(|c_1|r^{1-n})}{c_2g_2(|c_2|r^{1-n})}\Big) \quad\text{provided}\quad  c_1 c_2>0.
\eeq
Note that 
\beq\label{sstar2}
\lim_{r\to\infty} h(r)=s^* \eqdef f^{-1}\Big(\frac{c_1 a_{1}^0 }{c_2a_{2}^0 }\Big)\in (0,1).
\eeq
Let $\xi(r)=r^{1-n}\in (0,\infty)$. We rewrite $h(r)$ as 
\beq\label{hQ}
h(r)=f^{-1}\Big(Q(\xi(r)) \Big)\quad  \text{ where } \quad Q(\xi) = \frac {c_1g_1(|c_1|\xi)}{c_2g_2(|c_2|\xi)}\quad\text{for }\xi>0.
\eeq 


\begin{theorem}\label{genS}
If solution $S(r)$ of \eqref{IVP} exists in $[r_0,\infty)$, then there exists $R>r_0$ such that solution  $S(r)$ is monotone on $(R,\infty)$, and, consequently, $\lim_{r\to\infty} S(r)$ exists.
\end{theorem}
\begin{proof} 
If $c_1c_2\le 0$ then  all $r\ge r_0$ either $S'\ge 0$  or $S'\le 0$. Thus $S(r)$ is monotone on $[r_0,\infty)$. 

Consider the case $c_1c_2>0$. Then $h(r)$ in \eqref{hQ} exists. We rewrite $Q(\xi)$ as
\beq
Q(\xi)=\frac{c_1}{c_2} \cdot \frac{\sum_{i=0}^{m_1} a_i \xi^{\alpha_i} }{\sum_{i=0}^{m_2} b_i \xi^{\beta_i} }.
\eeq
where $a_i, b_i>0$, $0=\alpha_0<\alpha_1<\cdots< \alpha_{m_1}$, $0=\beta_0<\beta_1<\cdots< \beta_{m_2}$  .  

If $Q'\equiv 0$ then $h(r)\equiv s^*$ is an equilibrium. It is easy to see that if $s_0>(<)s^*$ then $S(r)>(<)s^*$ for all $r$, hence $S(r)$ is monotone on $r\in [r_0,\infty)$.

Now we consider $Q' \neq 0$. A simple calculation gives 
\begin{align*}
Q'(\xi)  
=\frac {c_1}{\xi c_2} (\sum_{i=0}^{m_2} b_i\xi^{\beta_i})^{-2} \sum_{i=1}^{m_3} A_i \xi^{\gamma_i},
\end{align*} 
where $m_3\ge 1$, $A_i\neq 0,  0<\gamma_1<\gamma_2<\cdots <\gamma_{m_3}  $.
Note that $Q'(\xi)$ has the same sign as $A_1$ for $\xi>0$ sufficiently small. Combining this with the fact $f'>0$, we have that $A_1 h'(r)<0$ for all $r>R$, where $R>0$ is a sufficiently large number. 

\textbf{Claim 1.} There is $\tilde R>R$ such that $S'(r)\ge 0$ on $(\tilde R,\infty)$ or $S'(r)\le 0$ on $(\tilde R,\infty)$.

Then the theorem's statements obviously follow Claim 1. 

To prove Claim 1 we consider the following cases.

\noindent {\bf Case 1:} $A_1<0$. Then $h(r)$ is increasing in $[R,\infty)$ and, hence, $h(r)<s^*$ for all $r\ge R$. 

Case 1A: $S(r)\ge h(r)$ for all $r>R$. Then $S'\ge 0$ for all $r>R$ or $S'\le 0$ for all $r>R$. 

Case 1B: There exists $R_1>R$ such that $S(R_1)<h(R_1)$.
     
+ Case 1B(i):  $F(r,S)>0 \Leftrightarrow S>h(r)$. Then $S'>0$ if $S(r)>h(r)$ and $S'<0$ if $S(r)<h(r)$. 
It is easy to see that $S(r)<h(R_1)\le h(r)$ for all $r>R_1$. 
Therefore $S'(r)<0$ for all $r>R_1$.
        
+ Case 1B(ii): $F(r,S)<0 \Leftrightarrow S>h(r)$. Then $S'<0$ if $S(r)>h(r)$ and $S'>0$ if $S(r)<h(r)$. 

\textbf{Claim 2.} $S(r)\le h(r)$ for all $r\ge R_1$ and hence $S'(r)\ge 0$ for all $r>R_1$.

Suppose Claim 2 is false. Then there is $R_2>R_1$ such that $ S(R_2)> h(R_2)$. There is $\tilde r\in (R_1 ,R_2)$ such that $S(\tilde r)=h(\tilde r)$. Hence, $S$ is decreasing on $(\tilde r , R_2)$, $S(R_2)\le S(\tilde r)=h(\tilde r)\le h(R_2)$. This is a contradiction.

\noindent{\bf Case 2:} $A_1>0$. Then $h(r)$ is decreasing in $[R,\infty)$ and $h(r)>s^*$ for all $r\ge R$  . 

Case 2A: $S(r)\le h(r)$ for all $r>R$. Then $S'\le 0$ for all $r>R$ or $S'\ge 0$ for all $r>R$.  

Case 2B: There exists $R_1>R$ such that $S(R_1)>h(R_1)$.
  
+ Case 2B(i):  $F(r,S)>0 \Leftrightarrow S>h(r)$.  Then $S'>0$ if $S(r)>h(r)$ and $S'<0$ if $S(r)<h(r)$. 
Similar to Case 1B(i), $h(r)<h(R_1)< S(r)$ for all $r>R_1$. 
Therefore $S'(r)>0$ for all $r>R_1$.     
          
+ Case 2B(ii):  $F(r,S)<0 \Leftrightarrow S>h(r)$.  Then $S'<0$ if $S(r)>h(r)$ and $S'>0$ if $S(r)<h(r)$. Similar to Case 1B(ii),  $S(r)\ge h(r)$ for all $r\ge R_1$. Therefore $S'(r)\le 0$ for all $r>R_1$.
 
From the above considerations, we see that Claim 1 holds true and the proof is complete. 
\end{proof}

Let $s_\infty=\lim_{r\to\infty} S(r)$ in Theorem \ref{genS}. Note that $s_\infty\in [0,1]$.

\begin{lemma}\label{special}
 For $n=2$ and $c_1^2+c_2^2 >0$, if $s_\infty$ is neither $0$ nor $1$ then $s_\infty$ must be $s^*$.
\end{lemma}
\begin{proof} Assume $s_\infty\ne 0,1$. We prove by contradiction. Suppose $s_\infty\not =s^*$. Then
\beq 
c_3\eqdef |F_2(s_\infty)a_2^0c_2 - F_1(s_\infty)a_1^0c_1| >0.
\eeq 
For any $R>r_0$, We write $S(r) =I_1(R)+I_2(R)$ where 
\beqs
I_1(R)= s_0+\int_{r_0}^R F(z,S(z))dz\quad \text{and}\quad I_2(R)=\int_R^r F(z,S(z))dz.
\eeqs
For sufficiently large $R$ and $r>R$ 
\beqs
|I_2(R)|=\int_R^r F(z,S(z))dz\ge \frac {c_3} 2\int_R^r z^{-1} dz=\frac {c_3} 2 (\ln r-\ln R).
\eeqs
Therefore
\beqs
|S(r)|\ge  \frac {c_3} 2(\ln r-\ln R) - I_1(R) \to \infty \text { as } r\to\infty.
\eeqs
Thus $S(r)$ is unbounded which contradicts the fact $S(r)\in (0,1)$. Hence $s_\infty=s^*$.
\end{proof}

Using Lemma \ref{special} we can drastically reduce the range of $s_\infty$ in case $n=2$.

\begin{theorem}\label{special2}
 Let $n=2$ and $c_1^2+c_2^2 >0$. Suppose $S(r)$ is a solution of \eqref{IVP} on $[r_0,\infty)$.

{\rm (i)} If $c_1\le 0$ and $c_2\ge 0$ then $s_\infty=1$.

{\rm (ii)} If $c_1\ge 0$ and $c_2\le 0$ then $s_\infty=0$.

{\rm (iii)} If $c_1,c_2<0$ then $s_\infty=s^*$.

{\rm (iv)} If $c_1,c_2>0$ then $s_\infty\in \{0,1,s^*\}$.
\end{theorem}
\begin{proof}
(i) In this case, $S'(r)>0$ for all r, hence $S(r)>s_0$. This implies $s_\infty\ne 0$. In addition, $s^*$ does not exist. Therefore, by Lemma \ref{special}, $s_\infty$ must be $1$.

(ii) The proof is similar to that of (i).

(iii) We have $F(r,S)<0$ for $S<h(r)$ and $F(r,S)>0$ for $S<h(r)$. Thus, it is easy to see that $s_\infty$ cannot be $0,1$. By Lemma \ref{special}, $s_\infty$ must be $s^*$.

(iv) This is a direct consequence of Lemma \ref{special}.
\end{proof} 

In general, we do not know the value of $s_\infty$ based on $s_0$.
However, in some particular cases, we can determine the range of $s_\infty$.

\begin{example}\label{exam}
We consider the following special $g_i$'s: 
\beq \label{2term}
g_i(u) =a_i+b_i u^\alpha\quad \text{ where } a_i>0,\ b_i>0,\text{ for } i=1,2 \text { and } \alpha>0.
\eeq 
 We have from \eqref{hQ} when $c_1c_2>0$ that  
 \beqs
 Q'(\xi) = \frac {c_1\Delta} { c_2(a_2+b_2|c_2|^\alpha \xi )^2 }
 \text{ with } \Delta  = a_2b_1|c_1|^\alpha - a_1b_2 |c_2|^\alpha. 
 \eeqs

We now detail the range of $s_\infty$ case by case.

\noindent\textbf{Case $n> 2$.}
\begin{itemize}
 \item[A.] $c_1,c_2>0$. 

\begin{itemize}
 \item[A1.] $\Delta<0$. 

(i) $s_0>s^*$. Then $s_\infty\in (s_0,1]$.

(ii) $h(r_0)\le s_0 \le s^*$. Then $s_\infty\in [0,1]$.

(iii) $s_0<h(r_0)$. Then $s_\infty \in [0,s_0)$.

\item[A2.] $\Delta>0$. 

(i) $s_0>h(r_0)$. Then $s_\infty\in (s_0,1]$.

(ii) $s^* \le s_0 \le h(r_0)$. Then $s_\infty\in [0,1]$.

(iii) $s_0<s^*$. Then  $s_\infty \in [0,s_0)$.

\item[A3.] $\Delta=0$. 

(i) $s_0>s^*$. Then $s_\infty\in(s_0,1]$.

(ii) $s_0=s^*$. Then $s_\infty=s^*$.

(iii) $s_0<s^*$. Then  $s_\infty \in [0,s_0)$.

\end{itemize}

\item[B.] $c_1,c_2<0$. 
\begin{itemize}
 \item[B1.] $\Delta<0$. 

(i) $s_0>s^*$. Then $s_\infty\in(h(r_0) ,s_0)$.

(ii) $h(r_0)\le s_0 \le s^*$. Then $s_\infty\in (h(r_0),s^*]$.

(iii) $s_0<h(r_0)$. Then $s_\infty \in (s_0,s^*]$.

 \item[B2.] $\Delta>0$. 

(i) $s_0>h(r_0)$. Then $s_\infty\in [s^*,s_0)$.

(ii) $s^* \le  s_0 \le h(r_0)$. Then $s_\infty\in [s^*,h(r_0))$.

(iii) $s_0<s^*$. Then  $s_\infty \in (s_0,h(r_0))$. 

\item[B3.] $\Delta=0$. 

(i) $s_0>s^*$. Then $s_\infty\in [s^*,s_0)$.

(ii) $s_0=s^*$. Then $s_\infty=s^*$.

(iii) $s_0<s^*$. Then  $s_\infty \in (s_0,s^*]$.

\end{itemize}

\item[C.] $c_1\le 0<c_2 $ or $c_1<0=c_2$. Then $s_\infty\in(s_0,1]$.

\item[D.] $c_2\le 0<c_1$  or $c_1=0> c_2$. Then $s_\infty\in[0,s_0)$.
\end{itemize}

Verifications of the cases above are presented in the Appendix.

\noindent\textbf{Case $n=2$.}  We use the analysis in A, which is still valid for $n=2$, to explicate the case $c_1,c_2>0$ in Theorem \ref{special2}. 
Let $s_m=\min\{h(r_0),s^*\}$ and $s_M=\max\{h(r_0),s^*\}$.
\begin{itemize}
 \item[(i)] $s_0>s_M$. Then $s_\infty=1$.

 \item[(ii)] $s_m\le s_0\le s_M$. Then $s_\infty\in \{0,1,s^*\}$.

 \item[(iii)] $s_0<s_m$. Then  $s_\infty=0$.
\end{itemize}

\end{example}

\myclearpage

 \section{Linearization}\label{linearization}
 
We study the linear stability of a steady state solution $(\U_1^*(\x),\U_2^*(\x), S_*(\x))$ of system \eqref{mainsys}.
The formal linearizion of system \eqref{mainsys} at  $(\U_1^*(\x),\U_2^*(\x), S_*(\x))$ is 
\begin{subequations}\label{linsys}
 \begin{align}
\label{s1}\sigma_t &=- {\rm div} \  {\mathbf v}_1 ,\\
\label{s2}\sigma_t &= {\rm div} \  {\mathbf v}_2,\\
\label{s3}\nabla \sigma &=F_2(S_*) \G'_2({\mathbf u}_2^*){\mathbf v}_2+ F'_2(S_*)\sigma \G_2({\mathbf u}_2^*) -\Big(F_1(S_*) \G'_1({\mathbf u}_1^*){\mathbf v}_1+F'_1(S_*)\sigma\G_1({\mathbf u}_1^*)\Big).
\end{align}
\end{subequations}

Above, the unknowns are $\sigma(\x,t)\in\R$, $\V_1(\x,t),\V_2(\x,t)\in\R^n$.
A solution $(\sigma,\V_1,\V_2)$ of \eqref{linsys} is considered as an approximation of the difference between a solution $(S(\x,t),\U_1(\x,t),\U_2(\x,t))$ of \eqref{mainsys} and the steady state $(\U_1^*(\x),\U_2^*(\x), S_*(\x))$ in \eqref{Sradial}. The system \eqref{linsys} is obtained by utilizing Taylor expansions in \eqref{mainsys} at $(\U_1^*,\U_2^*, S_*)$ with respect to variables $\U_1,\U_2, S$ and then neglecting non-linear  terms. In theory of ordinary differential equations, linearizion has direct connections with the stability of steady states. In PDE theory, this is not always the case. Nonetheless, in many scenarios, stability of the linearized equations lead to the stability of the original ones. In this article we only focus on the stability for the linearized system \eqref{linsys}.

We consider, particularly, the steady states obtained in the previous section, that is,
\beq \label{Sradial}
\U_1^*(\x) =c_1 |\x|^{-n}\x,\quad \U_2^*(\x)=c_2 |\x|^{-n}\x,\quad S_*(\x)=\hat S(|\x|),
\eeq
where $c_1,c_2$ are constants and $\hat S(r)$ is a solution of \eqref{IVP}.

Let ${\mathbf v} ={\mathbf v}_1+ {\mathbf v}_2.$ Adding equation \eqref{s1} to \eqref{s2} gives 
\beq\label{divv}
{\rm div} \ {\mathbf v} =0.
\eeq
Assume $ {\mathbf v} = \mathbf V(\x,t)\in \mathbb R^n$, where $\mathbf V(\x,t)$ is a given function.   
We have 
\beq\label{v1} 
\mathbf{v}_1 = \mathbf V-\mathbf{v}_2,
\eeq 
hence \eqref{s3} provides   
\beq\label{s31}
\nabla \sigma = \sigma\mathbf b + \Bm\mathbf v_2 - \mathbf c,
\eeq
where 
\begin{align}
\label{mtxB}
\Bm&= \Bm(\x)= F_2(S_*)\G'_2({\mathbf u}_2^*) + F_1(S_*)\G'_1({\mathbf u}_1^*),\\
\label{bdef}
\mathbf b &= \mathbf b(\x)= F'_2(S_*)\G_2({\mathbf u}_2^*)-F'_1(S_*)\G_1({\mathbf u}_1^*),\\ 
\label{cdef}
\mathbf c&=\mathbf c(\x,t)= F_1(S_*)\G'_1({\mathbf u}_1^*) \mathbf V(\x,t)  .
\end{align}
The $n\times n$ matrix $\Bm$ is invertible (see Lemma \ref{Bsym} below), and we denote its inverse by 
\beq \Am=\Am(\x)=\Bm^{-1}(\x).
\eeq
Solving for $\V_2$ from \eqref{s31}  we obtain 
\beq\label{v2}
 \mathbf {v}_2= \Am(\nabla \sigma - \sigma\mathbf b) + \Am \mathbf c.
\eeq  
Substituting \eqref{v2} into \eqref{s2} gives     
\beq
\begin{aligned}\label{sigt}
\sigma_t& = \nabla\cdot\Big[ \Am(\nabla \sigma -\sigma \mathbf b )\Big] + \nabla\cdot(\Am \mathbf c).
\end{aligned}
\eeq 
Then \eqref{sigt}, \eqref{v1} and \eqref{v2} is our linearized system for \eqref{mainsys}
at the steady state $(\U_1^*(\x),\U_2^*(\x), S_*(\x))$.

   
\begin{remark}\label{pertr}
In our approach, the total velocity $\bf V$ and hence the vector function $\bf c$ are supposed to be known, whereas the phase velocities $\mathbf v_i$ ($i=1,2$) are the unknowns.
Therefore, our results below can be considered as the qualitative study of the flow depending on the property of the total velocity.
Such restriction, however, is justified in practice or in case $\bf V$, as a perturbation, itself is radial.
In the latter consideration, by \eqref{divv}, $\mathbf V=\mathbf V(t)$ is totally determined by its boundary values.
\end{remark}  

We will focus on studying classical solutions of \eqref{sigt}. For such purpose, the maximum principle plays an important role. 
Although there is not an obvious maximum principle for \eqref{sigt},
we can convert it to an equation for which there is one. We proceed as follows.
Rewrite vector function $\mathbf b(\x)$ explicitly as 
\beq\label{vecb}
\mathbf b(\x) = \Big(F'_2(S_*(\x))g_2(\frac {|c_2|}{|\x|^{n-1}})\frac {c_2}{|\x|^n} -F'_1(S_*(\x))g_1(\frac {|c_1|}{|\x|^{n-1}})\frac {c_1}{|\x|^n} \Big)\x
                = \lambda(|\x|) \x,
\eeq
where 
\beq\label{lamr}
\lambda(r) =F'_2(\hat S(r))g_2(\frac {|c_2|}{r^{n-1}})\frac {c_2}{r^n} -F'_1(\hat S(r))g_1(\frac {|c_1|}{r^{n-1}})\frac {c_1}{r^n}. 
\eeq
By defining 
\beq\label{Lamddef}
\Lambda(\x) = \frac 12 \int_{r_0^2}^{|\x|^2} \lambda(\sqrt{\xi} )d\xi = \int_{r_0}^{|\x|} r\lambda(r)dr ,
\eeq 
we have for $\x\ne 0$ that
\beq\label{b&lamb}
\mathbf b(\x)=\nabla \Lambda(\x) . 
\eeq
Substituting this relation into \eqref{sigt} we obtain 
\begin{align*}
\sigma_t &= \nabla\cdot\Big[ \Am(\nabla \sigma - \sigma \nabla \Lambda  )\Big] + \nabla\cdot(\Am \mathbf c)
= \nabla\cdot\Big[  e^\Lambda \Am\nabla( e^{-\Lambda}\sigma)\Big]+ \nabla\cdot(\Am \mathbf c)\\
&=  e^\Lambda  \nabla\cdot\Big[ \Am\nabla( e^{-\Lambda}\sigma)\Big]+ e^\Lambda  \nabla\Lambda \cdot \Big[  \Am\nabla( e^{-\Lambda}\sigma)\Big]
+ \nabla\cdot(\Am \mathbf c).
\end{align*}
Let
\beq\label{w-sig}
w(\x,t)=e^{-\Lambda(\x)}\sigma(\x,t).  
\eeq
Then $w$ satisfies 
\beq
w_t=e^{-\Lambda}\sigma_t = \nabla\cdot\Big(  \Am\nabla w\Big) +  \nabla\Lambda\cdot \Am\nabla w   + e^{-\Lambda}\nabla\cdot(\Am \mathbf c).
\eeq
Using relation \eqref{b&lamb} again yields
\beq\label{neweq}
w_t- \nabla\cdot ( \Am\nabla w ) -   \mathbf{b} \cdot \Am \nabla w  =  e^{-\Lambda}\nabla\cdot(\Am \mathbf c).
\eeq 

For the velocities, we have from \eqref{v2} and \eqref{w-sig} that 
 \begin{align*} 
  \V_2 = \Am\big[\nabla (e^{\Lambda} w) -e^{\Lambda}w{\bf b}\big] + \Am{\bf c}
  = \Am\big[ e^{\Lambda} \nabla w +  w e^{\Lambda}\nabla \Lambda  -e^{\Lambda}w{\bf b}\big] + \Am{\bf c}.
 \end{align*}
Thus, 
 \beq\label{v2w} 
  \V_2= e^{\Lambda}\Am\nabla w +\Am{\bf c} .
 \eeq

We will proceed by studying \eqref{neweq} first and then drawing conclusions for $\sigma, \V_1, \V_2$ via the relations \eqref{w-sig}, \eqref{v2w} and \eqref{v1}.


In the following, we present some properties of $\Bm$, $\Am$ and $\bf  b$.
They have some structures and estimates which are crucial for our next sections.
These are based on the special form of the steady state $(\U^*_1,\U^*_2,S_*)$.

Denote by $\Idn$ the $n\times n$ identity matrix.
Consider $c_1^2+c_2^2>0$ and $\x\ne 0$.
We have for  $i=1,2$ that 
\beq\label{primeGi}
\G'_i(\U_i^*) =g_i(|\U_i^*|)\Idn +g'_i(|\U_i^*|)  \frac {\U_i^*(\U_i^*)^T}{|\U_i^*|}
= g_i(|c_i|\, |\x|^{1-n} )\Idn +g'_i(|c_i|\,|\x|^{1-n})  |c_i|\, |\x|^{-1-n} \x\x^T.
\eeq
Since these matrices are symmetric, so is $\Bm$. 
For each $i=1,2$ and arbitrary $\mathbf{z}\in \mathbb R^n$,   
\begin{align*}
\mathbf{z}^T \G'_i(\U_i^*) \mathbf{z} 
&= g_i(|c_i|\, |\x|^{1-n} )|\mathbf{z}|^2 +g'_i(|c_i|\,|\x|^{1-n})  |c_i|\, |\x|^{-1-n} |\x\cdot {\mathbf z}|^2.
\end{align*}
Define
\begin{align}
\label{betadef}
\beta &=\beta(\x)\eqdef \sum_{i=1}^2 F_i(S_*(\x))g_i(|c_i|\, |\x|^{1-n}),  \\
\label{gamdef}
\gamma &= \gamma(\x)\eqdef \sum_{i=1}^2 F_i(S_*(\x))g'_i(|c_i|\,|\x|^{1-n})  |c_i|\, |\x|^{1-n}.
\end{align}
Then
\begin{align}
\label{B2}
\beta |\mathbf{z} |^2 \le \mathbf{z}^T \Bm \mathbf {z} &\le (\beta+\gamma) |\mathbf{z} |^2 .
\end{align} 

The first inequality in \eqref{B2} proves that $\mathbf{z}^T \Bm \mathbf {z}>0$ for all $\mathbf{z}\neq 0$. Therefore, $\Bm$ is positive definite and hence it is invertible.  
Since  $\Bm$ is symmetric, so is its inverse $\Am$.  Thus, we have:

\begin{lemma}\label{Bsym}
For any $c_1^2+c_2^2>0$ and $\x\ne 0$, matrices $\Bm(\x)$ and $\Am(\x)$ are symmetric, invertible and positive definite.
\end{lemma}

Since matrix $\Bm$ is symmetric and positive definite, it has positive eigenvalues
$\lambda_1(\Bm)\le \lambda_2(\Bm)\le\cdots\le \lambda_n(\Bm)$.
We have
\beq\label{spec}
  \lambda_1(\Bm)=\min_{\mathbf z\neq 0} \frac{\mathbf z^T \Bm\mathbf z}{|\mathbf z|^2 }\quad \text{and}\quad
  \lambda_n(\Bm)=\max_{\mathbf z\neq 0} \frac{\mathbf z^T \Bm\mathbf z}{|\mathbf z|^2 }. 
\eeq
It follows from \eqref{spec} and  \eqref{B2} that 
\beq\label{eigB}
\beta \le \lambda_{1} (\Bm)\le \lambda_{n} (\Bm) \le \beta+\gamma.
\eeq
By the Spectral Theorem,
\beq\label{eigA}
\lambda_{1}(\Am)=\frac 1{\lambda_{n}(\Bm)} \ge \frac1{\beta+\gamma}
\quad \text{and}\quad
\lambda_{n}(\Am)=\frac 1{\lambda_{1}(\Bm)} \le \frac1{\beta}.
 \eeq 

We now consider $0<r_0\le |\x| <R_{\rm max}$. Let $a_0^{(i)}=g_i(0)$ for $i=1,2$, and define 
\begin{align}
d_0&=\min\{a_0^{(1)},a_0^{(2)}\} , 
  \quad  d_1=d_1(r_0)=\sum_{i=1}^2 g_i(|c_i|r_0^{1-n}),\\
d_2&=d_2(r_0)=\sum_{i=1}^2 g_i(|c_i|r_0^{1-n})|c_i|r_0^{1-n},
  \quad d_3=d_3(r_0)=\sum_{i=1}^2 g'_i(|c_i|r_0^{1-n})|c_i|r_0^{1-n},\\
d_4&=d_4(r_0)=d_1+d_3.
\end{align}
Then
\beq\label{bgest}
d_0 \sum_{i=1}^2 F_i(S_*(\x)) \le \beta(\x)\le d_1 \sum_{i=1}^2 F_i(S_*(\x))\quad\text{and}\quad
\gamma(\x)\le d_3 \sum_{i=1}^2 F_i(S_*(\x)).
\eeq
By \eqref{B2}, \eqref{eigA} and \eqref{bgest},
\beq\label{B3}
d_0 |\mathbf{z} |^2 \sum_{i=1}^2  F_i(S_*(\x)) \le \mathbf{z}^T \Bm(\x) \mathbf {z} \le d_4 |\mathbf{z} |^2 \sum_{i=1}^2  F_i(S_*(\x)),
\eeq 
\beq
\frac1{d_4\sum_{i=1}^2 F_i(S_*(\x))}\le \lambda_{1}(\Am)
\le \lambda_{n}(\Am)\le \frac1{d_0\sum_{i=1}^2 F_i (S_*(\x))}.
 \eeq 
Applying \eqref{spec} to matrix $\Am$, we have
 \beq\label{Aellip}
 \mathbf  z^T\Am(\x)\mathbf  z \ge \lambda_1(\Am)|\mathbf  z|^2 \ge \frac{|{\bf z}|^2}{d_4\sum_{i=1}^2 F_i(S_*(\x))}\quad \forall\bf z\in\R^n.
 \eeq
Denote by $|\Am|$ and $\|\Am\|_{\rm op}$ the Euclidean and operator norms of matrix $\Am$, respectively.
Then
 \beq\label{equiv}
 |\Am|\le c_0\|\Am\|_{\rm op}=c_0\lambda_{n}(\Am) ,
 \eeq       
for some constant $c_0>0$.  Thus,
 \beq\label{Abound}
 |\Am(\x)|\le \frac{c_0}{d_0\sum_{i=1}^2 F_i (S_*(\x))}\quad \forall |\x|\in [r_0,R_{\rm max}).
 \eeq       

For the boundedness of $\mathbf b$, we have 
\beq\label{best}
|\mathbf b(\x)| \le  \sum_{i=1}^2 \Big[ |F'_i(\hat S(|\x|)) | g_i( |c_i| |\x|^{1-n}) |c_i| |\x|^{1-n}\Big] 
\le d_2 \sum_{i=1}^2 |F'_i(\hat S(|\x|)) |\quad \forall |\x|\in [r_0,R_{\rm max}).
\eeq

From \eqref{Lamddef} and \eqref{lamr},  
\beq\label{Lambdform}
\Lambda(\x) = \int_{r_0}^{|\x|} r\lambda(r)dr 
=\int_{r_0}^{|\x|}\Big[ F'_2(\hat S(r))G_2(c_2 r^{1-n})  - F'_1(\hat S(r))G_1( c_1 r^{1-n})\Big]dr.
\eeq
Then
\beq\label{Lamest1}
|\Lambda (\x)|\le d_2\int_{r_0}^{|\x|} \Big[ |F'_1(\hat S(r))| + |F'_2(\hat S(r))|\Big]dr \quad \forall |\x|\in [r_0,R_{\rm max}).
\eeq

Also, matrix $\Bm$ has the following special property:
\beq\label{Bxx}
\Bm(\x)\x
= \sum_{i=1}^2 \Big\{ F_i ( \hat S(|\x|))\left[ g_i(|c_i| |\x|^{1-n} ) +g'_i(|c_i| |\x|^{1-n})  |c_i| |\x|^{1-n} \right] \Big\} \x
=\phi(|\x|)\x,
\eeq
where 
\beq\label{defphi}
\phi (r)=  \sum_{i=1}^2 F_i ( \hat S(r))\left[g_i(|c_i| r^{1-n} ) +g'_i( |c_i| r^{1-n})  |c_i| r^{1-n}\right].
\eeq
Since $g'_i\ge 0$,
\beq\label{phiest0}
\phi (r) \ge   d_0  [F_1( \hat S(r))+F_2( \hat S(r))] \quad  \forall r \in [r_0,R_{\rm max}).
\eeq
Since $g_i(s)$ and $g_i'(s)s$ are increasing on $[0,\infty)$, we have
\beq\label{phiest1}
 \phi(r)\le d_4[F_1(\hat S(r))+F_2(\hat S(r))] \quad\forall r \in [r_0,R_{\rm max}).
\eeq

We now discuss the regularity of the involved functions. For $D\subset \R^n\times\R$, we define class $C_{\x}^m(D)$ as the set of functions $f(\x,t)\in C(D)$ whose partial derivatives with respect to $\x$ up to order $m$ are continuous in $D$. The class $C_{t}^m$ is defined similarly and $C_{\x,t}^{m,m'}=C_{\x}^m\cap C_{t}^{m'}$. 

Note that
\beq\label{derA}
\frac{\partial \Am}{\partial x_i}=-\Am \frac{\partial \Bm}{\partial x_i} \Am.
\eeq
By definitions \eqref{mtxB}, \eqref{bdef}, \eqref{cdef} and relation \eqref{derA}, we easily obtain:

\begin{lemma}\label{smooth}
Assume $F_1,F_2\in C^m((0,1))$ for some $m\ge 1$. Let $R\in (r_0,R_{\rm max})$ and
denote $$\mscrO=\{\x:r_0<|\x|<R\}.$$

{\rm (i)} Then $\Bm,\Am\in C^m(\bar \mscrO)$, $\mathbf b\in C^{m-1}(\bar \mscrO)$
and $\Lambda\in C^m(\bar \mscrO)$.

{\rm (ii)} If, in addition, $\mathbf V\in X(\mscrO\times (0,\infty))$ then $\mathbf c\in X(\mscrO\times (0,\infty))$, where $X$ can be $C^m$ 
or $C_{\x}^m$ or $C_{t}^m$.
\end{lemma}

\myclearpage
\section{Case of bounded domain}\label{bounded}

In this section, we study the linear stability of the  obtained  steady flows in section \ref{Steady} on bounded domains.
More specifically, we investigate the stability of the trivial solution for the linearized system \eqref{linsys}. 
The key instrument in proving the asymptotic stability is a Landis-type lemma of growth (see \cite{LandisBook}). To prove such a lemma we  use specific structures of the coefficients of equation  \eqref{neweq}  to construct singular  sub-parabolic functions. These are motivated by the so-called $F_{s,\beta}$ functions introduced in \cite{LandisBook}.

  

Let $r_0>0$ be fixed throughout.
We consider in this section an open, bounded set $U$ in $\R^n\setminus \bar B_{r_0}$.
We fix $R>0$ such that $U\subset \mscrU \eqdef  B_R\setminus \bar B_{r_0}$.
Denote $\Gamma=\partial U$, $D=U\times(0,\infty)$ and $\mscrD=\mscrU\times (0,\infty)$.

We consider a steady state $(u^*_1(\x),u^*_2(\x),S_*(\x))$ as in \eqref{Sradial} with $c_1^2+c_2^2>0$. 
Recall that \eqref{sigt}, \eqref{v1} and \eqref{v2} is our linearized system for \eqref{mainsys}.
We study the equation for $\sigma(\x,t)$ first. More specifically, we study the following initial-boundary value problem (IBVP):
  \beq\label{sigEq}
  \begin{cases}
  \sigma_t = \nabla\cdot\Big[ \Am(\nabla \sigma -  \sigma {\bf b})\Big] + \nabla\cdot(\Am \mathbf c) & \text { on } U\times (0,\infty),\\
   \sigma = g(\x,t) &\text { on } \Gamma \times(0, \infty),\\
    \sigma =\sigma_0(\x)  &\text { on } U \times\{t=0\}.    
   \end{cases}
   \eeq

Regarding the initial and boundary data in \eqref{sigEq}, we always assume that
\beq \label{datacond}
\sigma_0\in C(\bar U),\  g\in C(\Gamma\times[0,\infty))\text{ and } \sigma_0(\x)=g(\x,0) \text{ on }\Gamma.
\eeq

Assume that 
\beq\label{Sassump}
0<\underline{s}\le \hat S(r)\le \bar s<1 \quad \forall r\in [r_0,R],
\quad\text{where } \underline{s} \text{ and }\bar s \text{ are constants.}
\eeq
Assumption \eqref{Sassump} is valid for any solution $\hat S$ in Theorem \ref{localODE} with $R_{\rm max}>R$, in particular, when $R_{\rm max}=\infty$ as in Theorem \ref{globalODE}.
Under constraint \eqref{Sassump} and Assumptions A and B, we easily see the following facts.
Let 
\beq \label{mudef}
\mu_1=\sum_{i=1}^2 \max_{\underline{s}\le s\le \bar s} F_i(s), \quad
\mu_2=  \sum_{i=1}^2  \min_{\underline s \le s \le \bar s} F_i(s),\quad
\mu_3=\sum_{i=1}^2 \max_{\underline s \le s \le \bar s} |F'_i(s)|.
\eeq
Then $\mu_1$, $\mu_2$ and $\mu_3$ are positive numbers.

 From \eqref{Aellip} and \eqref{Sassump} follows that 
\beq\label{Aellip1}
{\bf z}^T\Am(\x){\bf z}\ge \frac{|{\bf z}|^2}{C_0}\quad \forall\x\in \bar \mscrU,\ \mathbf z\in\R^n,
\eeq
where  $C_0=d_4 \mu_1$.

From \eqref{Abound}, \eqref{best} and \eqref{Sassump}, we get
\beq\label{Abbound1}
 |\Am(\x)|\le \frac{c_0}{C_1}\quad\text{and}\quad \quad  |\mathbf b (\x)|\le C_2\quad\forall \x\in \bar \mscrU,
\eeq
where $c_0$ is in \eqref{equiv}, $C_1=d_0\mu_2$ and $C_2=d_2\mu_3$.

For the smoothness, by Lemma \ref{smooth},
\beq \label{smoooth1}
\Bm,\Am\in C^1(\bar \mscrU)\quad\text{and}\quad \mathbf b\in C(\bar \mscrU).
\eeq


First, we consider the the existence of classical solutions of \eqref{sigEq}.
We use the known result from theory of linear parabolic equations in  \cite{IKO}. This will require certain regularity of the coefficients of \eqref{sigEq}. Those requirements, in turn, can be formulated in terms of functions $F_1$ and $F_2$, thanks to  Lemma \ref{smooth}. 

\textbf{Condition} \condEa. $F_1,F_2\in C^7((0,1))$ and $V\in C_{\x}^6(\bar D)$; $V_t \in C_{\x}^3(\bar D)$.


\begin{theorem}[\cite{IKO}]\label{eu}
Assume \condEa, then there exists a unique solution $\sigma\in C(\bar D)\cap C^{2,1}_{\x,t}(D)$ of problem \eqref{sigEq}.
\end{theorem}

Note that we did not attempt to optimize Condition \condEa. As seen below, the study of qualitative properties of solution $\sigma$ will require much less stringent conditions than \condEa.

Now we turn to the stability, asymptotic stability and structural stability issues.
Our main tool is the maximum principle. As discussed in the previous section, we use the transformation \eqref{w-sig} to convert the PDE in \eqref{sigEq} to a more convenient form \eqref{neweq}.
Define the differential operator on the left-hand side of \eqref{neweq} by
\beq\label{Lw}
\mathcal Lw =\partial_t w- \nabla\cdot (\Am\nabla w) - {\bf b} \cdot \Am\nabla w.
\eeq
Corresponding to \eqref{sigEq}, the IBVP for $w(\x,t)$ is
  \beq\label{nonhom}
   \begin{cases}
   \mathcal L w = \fo & \text {in } U\times (0,\infty),  \\
    w(\x,0) =w_0(\x)& \text {in } U,\\ 
     w(\x,t)=  G(\x,t) &\text {on } \Gamma\times(0,\infty),
   \end{cases}
\eeq
where $w_0(\x)$ and $G(\x,t)$ are given initial data and boundary data, respectively, and $\fo(\x,t)$ is a known function.
We will obtain results for  solution $w$ of \eqref{nonhom}  and then reformulate them in terms of solution $\sigma$ of the original problem \eqref{sigEq}.

Since the existence and uniqueness issues are settled in Theorem \ref{eu}, our main focus now is the qualitative properties of solution $w$ of \eqref{nonhom}.
For these, we only need properties  \eqref{Aellip1}, \eqref{Abbound1}, the special structure of  equation \eqref{sigEq}, and the assumption that the classical solution in $C(\bar D)\cap C^{2,1}_{\x,t}(D)$ already exists. The fine properties of the solutions obtained below have their own merit in the theory of linear parabolic equations.

It follows from \eqref{Aellip1} and \eqref{Abbound1} that the maximum principle holds for any classical solution of $\mathcal L w \le (\ge) 0$ in $D$. To obtain better estimates for solutions, especially as $t\to\infty$, we use the  following barrier function.
Define
\beq\label{subPar}
 W(\x,t)= \begin{cases} t^{-s} e^{-\frac{\varphi (\x)}t}  &\text{if }  t > 0,\\
 0 & \text {if } t\le 0,
 \end{cases}
\eeq
where the number $s>0$  and the function $\varphi(\x)>0$ will be decided later.
Then
\beqs
 \mathcal L W = t^{-s-2}e^{-\frac\varphi t} \Big\{ t\big(-s+\nabla\cdot(\Am\nabla \varphi)+ {\bf b}\cdot \Am\nabla \varphi  \big)  +\varphi - (\Am\nabla \varphi) \cdot \nabla \varphi   \Big\}.
 \eeqs
Thus, $\mathcal L W\le 0$  if  
 \beq\label{sphicond}
 s\ge \nabla\cdot(\Am\nabla \varphi)+ {\bf b}\cdot \Am\nabla \varphi \quad\text{and}\quad
 \varphi \le (\Am\nabla \varphi) \cdot \nabla \varphi .
 \eeq
We will choose $\varphi$ to satisfy
\beq\label{gradA}
\Am\nabla \varphi = \kap \x,
\eeq
where $\kap$ is a positive constant selected later.
Equivalently, with the use of \eqref{Bxx},
\beq\label{gradphi}
\nabla \varphi =\kap  \Am^{-1}\x =\kap  \Bm\x =\kap \phi(|\x|)\x, 
\eeq
where $\phi (r)$ is defined by \eqref{defphi}.
By \eqref{phiest0}, \eqref{Sassump} and \eqref{mudef},
\beq\label{Cone}
\phi(r)\ge d_0 \mu_2=C_1 \quad\text{for } r_0\le r\le R.
\eeq
By  \eqref{phiest1}, \eqref{Sassump} and \eqref{mudef},
\beq\label{Czero}
\phi(r)\le d_4 \mu_1=C_0 \quad\text{for } r_0\le r\le R.
\eeq

Define for $\x\in \bar \mscrU$ the function
\beq\label{phi1}
\varphi(\x) = \kap\Big(\varphi_0+\int_{r_0}^{|\x|} r \phi(r) dr\Big),\quad \text{where }
\varphi_0= \frac{C_0r_0^2}2 \text{ and } \kap = \frac{C_0}{2C_1 }.
\eeq 
Then $\varphi(\x)$ satisfies both equations \eqref{gradA} and \eqref{gradphi}.
We have for $\x\in \bar\mscrU$ that
\beq\label{phibound}
0< \varphi ( \x)\le \kap \Big(\varphi_0+C_0 \int_{r_0}^{|\x|} r dr\Big)=\frac{\kap C_0} 2  |\x|^2.
\eeq
  Applying \eqref{gradA}, \eqref{gradphi}, and then \eqref{B3} and \eqref{mudef} we obtain  
\beq\label{Appbound}
(\Am\nabla \varphi) \cdot \nabla \varphi =\kap^2 \x^T \Bm\x
 \ge d_0 \kap^2\Big( \sum_{i=1}^2 F_i(\hat S(|\x|)) \Big)|\x|^2
 \ge d_0 \kap^2 \mu_2|\x|^2
= \kap^2 C_1|\x|^2=\frac{\kap C_0} 2  |\x|^2
\eeq
Then we have from \eqref{phibound} and \eqref{Appbound} that
$\varphi \le (\Am\nabla\varphi) \cdot \nabla \varphi$ in $\mscrU$, which is the second requirement in \eqref{sphicond}.
On the other hand, by \eqref{gradA} and \eqref{Abbound1},
\beq\label{bound0}
 \nabla\cdot(\Am\nabla \varphi)+ {\bf b}\cdot \Am\nabla \varphi =\kap(\nabla\cdot \x + {\bf b}\cdot \x) \le \kap(n+ C_2 R).
\eeq
Select 
\beq\label{s}
s=s_R\eqdef \kap(n+ C_2 R).  
\eeq
Then we have $s\ge \nabla\cdot(\Am\nabla \varphi)+ \mathbf b\cdot (\Am\nabla \varphi)$ in $\mscrU$, which is the first requirement in \eqref{sphicond}.
Thus, we obtain $\mathcal LW\le 0$ in $\mscrU\times (0,\infty)$. 
For further references, we formulate this as a lemma.

\begin{lemma}\label{WsubPar}
With parameter $s=s_R$ selected as in \eqref{s} and function $\varphi$ defined by \eqref{phi1}, the function $W(\x,t)$ in \eqref{subPar} 
belongs to $C_{\x,t}^{2,1}(\mscrD)\cap C(\bar{\mscrD})$ and satisfies $\mathcal L W\le 0$ in $\mscrD$. 
\end{lemma}
Above, the regularity of $W(\x,t)$ follows the fact that $\varphi(\x)\ge \kap\varphi_0>0$ for $\x\in \bar \mscrU$.

\medskip
We now establish this section's key lemma of growth.
We fix $s=s_R$ by \eqref{s} and also the following two parameters 
\beq
q=\frac {\kap C_0} {2s}\quad\text{and}\quad \eta_0 =\Big(\frac{r_0}R\Big)^{2s},
\eeq
and denote $D_1=U\times (0,qR^2]$. 
\begin{lemma}[Lemma of growth in time]\label{grow}
Assume $w(\x,t)\in C_{\x,t}^{2,1}(D_1)\cap C (\bar D_1)$.
If 
\beq\label{assump1}
\mathcal L w\le 0  \text{ on }  D_1
\quad \text { and } \quad 
w\le  0 \text{ on } \Gamma \times (0, qR^2),
\eeq
then
 \beq\label{a11}
\max\{ 0,\sup_U w(\x,qR^2) \} \le \frac 1{1+\eta_0} \max\{0, \sup_U w(\x,0) \}.
\eeq 
%
\end{lemma}
\begin{proof}
(i) Let $M=\max\{0,\sup_{\bar D_1} w\}$. By \eqref{assump1} and maximum principle, we have
\beq\label{Mrel}
M=\max\{0,\sup_{\bar U} w(\x,0)\}.
\eeq
Let $W(\x,t)$ be as in \eqref{subPar} and define the auxiliary function 
 \beqs
 \tilde W(\x,t) =M[1-\eta W(\x,t)],
 \eeqs 
 where constant $\eta>0$ will be specified later. Our intention is to prove that  \beq \label{domcomp}
\tilde W(\x,t)\ge w(\x,t)\quad \text{for all } (\x,t)\in \bar D_1.
\eeq 
By Lemma~\ref{WsubPar}, $\mathcal L W\le 0$ in  $D_1$, hence,  $\mathcal L \tilde W \ge0$ in  $D_1.$   
By maximum principle, it suffices to show that
\beq\label{bc-comparison}
\tilde W(\x,t)\ge w(\x,t) \quad \text{for all }  (\x,t)\in \partial_p D_1=\big[\bar U\times \{0 \} \big]\cup \big[\Gamma \times (0, qR^2]\big ].  
\eeq

On the base  $\bar U\times\left\{0\right\}$, function $W(\x,0)$ vanishes, hence,
\beqs
\tilde W(\x,0)=M\ge w(\x,0).  
\eeqs

On the side boundary $\Gamma \times (0, qR^2]$, additional analysis is required. 
First observe for $\x\in \bar \mscrU$ that
$\varphi(\x) \ge \kap\varphi_0=\frac{\kap C_0r_0^2}2$.
Therefore, 
\beq\label{Wtilde}
\tilde W(\x,t)= M \left[1- \eta t^{-s} e^{-\frac {\varphi(\x)}t} \right] \ge M\left[1- \eta t^{-s} e^{-\frac {\kap C_0r_0^2} {2t}} \right]\quad \text{in } \bar \mscrU\times [0,\infty).
\eeq
Let $h_0(t) =t^{-s} e^{-\frac {\kap C_0r_0^2} {2t}} $ for $t\ge 0$. 
By elementary calculations, the maximum of $h_0(t)$ is attained at $t_0=\frac{\kap C_0r_0^2}{2s}$.
By letting 
\beq\label{eta} 
\eta=\frac1{\max_{[0,\infty)} h_0(t)}=\frac1{h_0(t_0)}=\Big(\frac{e\kap C_0r_0^2}{2s} \Big)^{s},\eeq 
we get from \eqref{Wtilde} that
$\tilde W(\x,t)\ge M[1- \eta h_0(t_0)]=0$ in $\bar \mscrU\times [0,\infty)$.
Particularly,
 \beqs
 \tilde W(\x,t)\ge 0\ge w(\x,t)\quad \text{on } \Gamma\times (0,qR^2].
  \eeqs

Thus, the comparison in \eqref{bc-comparison} holds and, therefore, \eqref{domcomp} is proved.

We now estimate $\tilde W(\x,t)$. By \eqref{phibound}, for $(\x,t)\in D$ we have
\begin{align*}
\tilde W(\x,t) 
     \le M \left[1- \eta  t^{-s} e^{ - \frac {\kap C_0|\x|^2}{2t} }\right]
      \le M \left[1- \eta   t^{-s} e^{ - \frac {\kap C_0R^2}{2t} }\right].
\end{align*}
Let $h_1(t)=t^{-s} e^{ - \frac {\kap C_0R^2}{2t}}$ for $t>0$.
Then $t_1=\frac {\kap C_0R^2}{2s}=qR^2$ is the critical point and
\beqs
h_1(t_1)=(qR^2)^{-s} e^{ - \frac {\kap C_0}{2 q} } \ge \Big(\frac {2s}{e\kap C_0R^2}\Big)^s  . 
\eeqs
Letting $t=t_1$ in \eqref{domcomp}, we have
\beq
w(\x,t_1)\le \tilde W(\x,t_1) \le M \left[1- \eta \Big(\frac {2s}{e\kap C_0R^2}\Big)^s  \right]=M(1-\eta_0)\le \frac{M}{1+\eta_0}, 
\eeq
and, hence, \eqref{a11} follows.  
\end{proof}

Using Lemma  \ref{grow},  we show the decay, as $t\to\infty$, of solution $w(\x,t)$ of the IBVP \eqref{nonhom} in the homogeneous case, i.e., when $\fo\equiv 0$ and $G\equiv 0$.

\begin{proposition}[Homogeneous problem]\label{Edecay} 
Assume $w(\x,t)\in C_{\x,t}^{2,1}(D)\cap C(\bar{D}) $ satisfies
\beq\label{Lvanish}
\mathcal L w=0  \text{ in } D\quad\text{and} \quad  w= 0 \text { on } \Gamma\times(0,\infty).
\eeq
Then    
\beq\label{ubound}
 -e^{-\eta_1t}  \inf_U |w(\x,0)|   \le w(\x,t)  \le (1+\eta_0)e^{-\eta_1t}  \sup_U |w(\x,0)| \quad \forall (\x,t)\in D,   
\eeq
  where $\eta_1=\frac{\ln(1+\eta_0)}{qR^2}$. 
\end{proposition}
\begin{proof}
Let $k\in \N$.
Applying Lemma~\ref{grow} with $D_1$ being replaced by $U\times (T_{k-1},T_k]$ gives 
\beqs
 \max\{0, \sup_U w(\x,kqR^2) \}\le  \frac 1 {1+\eta_0} \max\{ 0,\sup_U w(\x,(k-1)qR^2) \}. 
\eeqs 
By induction in $k$, we obtain  
\beq\label{w+} 
   \max\{ 0,\sup_U w(\x,kqR^2)\} \le \frac 1 {(1+\eta_0)^k} \max\{ 0,\sup_U w(\x,0)\}. 
\eeq 
Now applying \eqref{w+} to function $-w$ instead of $w$, we obtain  
 \beq\label{w-} 
  \min\{ 0, \inf_U w(\x,kqR^2)\} \ge \frac 1 {(1+\eta_0)^k} \min\{ 0,\inf_U w(\x,0)\}. 
\eeq 

For any $t>0$, there is an integer $k\ge 0$ such that  $t\in (T_k, T_{k+1}]$ where
$ T_k =k qT^2 $. By \eqref{Lvanish} and maximum principle for domain $U\times(T_k,T_{k+1}]$, and then using \eqref{w+} we have 
\begin{align}\label{Rbound}
w(\x,t)& \le \max\{ 0,\sup_U w(\x,T_k)\} \le (1+\eta_0)^{-k}\max\{ 0,  \sup_U w(\x,0)\}\nonumber\\
&=(1+\eta_0)e^{-\eta_1 T_{k+1}}\sup_U |w(\x,0)|\le(1+\eta_0)e^{-\eta_1t}\sup_U |w(\x,0)|.
\end{align}
Similarly, using \eqref{w-} instead of \eqref{w+} we have
\begin{align}\label{Lbound}
w(\x,t) &\ge \min\{ 0,\inf_U w(\x,T_k)\} \ge (1+\eta_0)^{-k}\min\{ 0, \inf_U w(\x,0)\}\nonumber\\
&\ge -e^{-\eta_1 T_k}\inf_U |w(\x,0)|
\ge -e^{-\eta_1t}\inf_U |w(\x,0)|.
\end{align}
Therefore, \eqref{ubound} follows \eqref{Rbound} and \eqref{Lbound}.
\end{proof}


Next, we consider the non-homogeneous case for the IBVP \eqref{nonhom}. Similar to \eqref{datacond}, we always consider
\beq \label{wdata}
w_0\in C(\bar U),\  G\in C(\Gamma\times[0,\infty))\text{ and } w_0(\x)=G(\x,0) \text{ on }\Gamma.
\eeq

\begin{proposition}[Non-homogeneous problem] \label{Allt}
Assume $f_0\in C(\bar D)$ and  
\beq\label{smalldata}
\deltao \eqdef \sup_{U\times (0,\infty) } |\fo(\x,t)| +\sup_{\Gamma\times (0,\infty)} |G(\x,t)|   <\infty
\eeq
There is a positive constant $C$ such that 
if $w(\x,t)\in C_{\x,t}^{2,1}(D)\cap  C(\bar D)$ is a solution of \eqref{nonhom}, then
\beq\label{main0}
| w(\x,t)| \le C\big[ e^{-\eta_1t}\sup_U |w_0(\x)|+  \deltao\big] \quad \forall (\x,t)\in D,
\eeq
where  $\eta_1>0$ is defined in Proposition \ref{Edecay}. 
\end{proposition}
\begin{proof}  
Denote $ T_k =k qR^2 $ for any integer $k\ge 0$.
Let $k\in \N$ and 
\beq\label{v-w-rel}
 v_k(\x,t) = w(\x,t) - \deltao(t-T_{k-1}+1) \quad \text{ for }(\x,t)\in \bar U\times [T_{k-1},T_k].
\eeq
Then $v_k$ satisfies
\beqs
\mathcal L v_k =\mathcal L w  -\deltao  \mathcal = \fo -\deltao \le  0 \quad \text{in } U\times (T_{k-1},T_k],
\eeqs
and
\beqs
v_k(\x,t)\le 0 \text { on } \Gamma\times(T_{k-1},T_k).
\eeqs
Applying Lemma~\ref{grow} to function $v_k$, we have
 \beq
\max\{0,\sup_U v_k(\x,T_k)\} \le \frac 1{1+\eta_0} \max\{0, \sup_U v_k(\x,T_{k-1})\}.
\eeq
Note that $v_k(\x,T_k)=w(\x,T_k)-\deltao(qR^2+1)$ and $v_k(\x,T_{k-1})=w(\x,T_{k-1})-\deltao\le w(\x,T_{k-1})$.
Hence,
\begin{align*}
\max\{0,\sup_U w(\x,T_k)\} 
&\le \max\{0,\sup_U v_k(\x,T_k)\}  + \deltao(qR^2+1)\\
&\le \frac 1{1+\eta_0}  \max\{0,\sup_U w(\x,T_{k-1})\} + \deltao(qR^2+1).
\end{align*}
Iterating this inequality gives
 \beq\label{reqursiv-w}
 \begin{aligned}
\max\{ 0,\sup_U w(\x,T_k)\}& \le \frac 1{(1+\eta_0)^k}  \max\{0,\sup_U  w(\x,0)\} 
+ \deltao(qR^2+1) \sum_{j=0}^{k-1}\frac{1}{(1+\eta_0)^j}\\
 &\le \frac 1{(1+\eta_0)^k} \max\{0, \sup_U  w(\x,0)\} + \frac{\deltao(1+qR^2)(1+\eta_0)}{\eta_0}. 
\end{aligned}  
\eeq


By using the relation \eqref{v-w-rel} between $v_k(\x,t)$ and $w(\x,t)$, maximum principle for function $v_k(\x,t)$, and estimate \eqref{reqursiv-w},  we have for any $t\in [T_{k-1},T_k]$ with $k\ge 1$ that  
\begin{align*}
 w(\x,t) 
 & \le v_k(\x,t)+\deltao(1+qR^2)
\le \max\{ 0,\sup_U w(\x,T_{k-1})\} +\deltao(1+qR^2)\\
& \le (1+\eta_0)^{-k+1}\max\{ 0,  \sup_U  w(\x,0)\}+\frac{\deltao(1+qR^2)(1+\eta_0)}{\eta_0}+\deltao(1+qR^2)\\
&\le(1+\eta_0)^{-\frac t{qR^2}+1}\sup_U | w(\x,0)|+\frac{2\deltao(1+qR^2)(1+\eta_0)}{\eta_0}.
\end{align*}
Therefore,
\beq
w(\x,t) \le C\big[ e^{-\eta_1 t} \sup_U | w(\x,0)|+\deltao \big].
\eeq
Similarly, we obtain the same estimate for $(-w)$ and hence, \eqref{main0} follows.
\end{proof}

For the asymptotic behavior of $w(\x,t)$ as $t\to\infty$, we have the following.
\begin{corollary}\label{Larget} 
Assume $f_0\in C(\bar D)$ is bounded and 
\beq 
\deltatw \eqdef \limsup_{t\to\infty} \left[\sup_{\x\in U}|\fo(\x,t))| +\sup_{\x\in\Gamma}|G(\x,t)|\right]<\infty.
\eeq
There exists $C=C(\eta_0,q,R,M)>0$ such that if $w(\x,t)\in C_{\x,t}^{2,1}(D)\cap  C(\bar D)$  solves \eqref{nonhom},
then
\beq\label{limsup0}
\limsup_{t\to\infty} \left[\sup_{\x\in U}|w(\x,t)|\right] \le C\deltatw.
\eeq 
\end{corollary}
\begin{proof}
Note that 
\beqs
\sup_{U}|w_0(\x)|+ \sup_{D} |\fo(\x,t)| +\sup_{\Gamma\times (0,\infty)} |G(\x,t)| <\infty.
\eeqs
Then by Proposition \ref{Allt}, $w(\x,t)$ is bounded on $\bar D$.
Let $\varep>0$. From our assumption there is $t_0>0$ such that
$$
\sup_{U\times[t_0,\infty)}|\fo(\x,t))| +\sup_{\Gamma\times[t_0,\infty)}|G(\x,t)|< \deltatw+\varep.
$$
Applying Lemma~\ref{Allt} to the domain $U\times (t_0,\infty)$ we obtain  
\beq\label{small-w}
 |w(\x,t)| \le C[e^{-\eta_1(t-t_0)}\sup_{\x\in U} |w(\x,t_0)|+ \deltatw+\varep].
\eeq
Therefore, passing $t\to\infty$ and then $\varep\to 0$ in \eqref{small-w} yields \eqref{limsup0}. 
\end{proof}

Next, we estimate $|\nabla w(\x,t)|$ by using Bernstein's technique (c.f. \cite{IKO}).

\begin{proposition} \label{gradw} 
Assume $f_0\in C(\bar D)$, $\nabla f_0 \in C(D)$, \eqref{smalldata} and 
\beq
\deltafo \eqdef \sup_{D} |\nabla f_0|<\infty.
\eeq  
For  any $U'\Subset U$ there is $\tilde M>0$ such that 
if $w(\x,t)\in C_{\x,t}^{2,1}(D)\cap  C(\bar D)$ is a solution of \eqref{nonhom} that also satisfies $w\in C_{\x}^3(D)$ and $w_t\in C_{\x}^1(D)$,
then
\beq\label{Gradbound}
 |\nabla w (\x,t) |  \le  \tilde M\Big[1+\frac 1{\sqrt t}\Big] \Big[e^{-\eta_1t}\sup_U |w(\x,0)|+\deltao+ \sqrt \deltafo\Big]
\quad \forall (\x,t)\in U'\times (0,\infty).
\eeq
\end{proposition}
\begin{proof}
Note that $\nabla w\in C_{\x,t}^{2,1} (D)$. By using finite covering of compact set $U'$, it suffices to prove  \eqref{Gradbound} for $\x$ in some ball inside $U$. 
Consider a ball $B_{{\delta}}(\x_*) =\{\x:|\x-\x^*|\le \delta \}\Subset U$ with some $\x_*\in U$ and $\delta>0$. Let $t_0>0$, define in the cylinder $G_{\delta} \eqdef B_{\delta}(\x_*) \times (t_0, 1+t_0]$ the following auxiliary function 
\beq 
\tilde w(\x,t) = \tau\Phi(\x) |\nabla w|^2 + N w^2 +N_1 (1+t_0-t),
\eeq 
where 
\beq
\tau= t-t_0\in (0,1], \quad \Phi(\x) = (\delta^2- |\x-\x_*|^2)^2.
\eeq
The numbers $N, N_1\ge 0$ will be chosen later. 
We rewrite the operator $\mathcal L$ as
\beq
\mathcal L w=w_t - \sum_{i,j=1}^n a_{ij}(\x) \partial_i\partial_j w - \mathbf{\tilde b}\cdot \nabla w,
\eeq
where $\mathbf{\tilde b}(\x)=(\tilde b_1,\tilde b_2,\ldots,\tilde b_n) \eqdef \nabla \cdot \Am+\Am\mathbf b$. 
Then following the calculations in Theorem 1 of section 2 on page 450 in \cite{IKO} we have 
\beq\label{Ltilw}
\begin{aligned}
\mathcal L \tilde w&\le 
2\tau\Phi \Big\{ \sum_{i,j,k=1}^n \dx{a_{ij}}{x_k}\dx{w}{x_k} \ddx{w}{x_i}{x_j}  
         +   \sum_{i,k=1}^n \dx{\tilde b_{i}}{x_k}\dx{w}{x_k} \dx{w}{x_i}
 - \sum_{i,j,k=1}^n a_{ij}\ddx{w}{x_k}{x_i}\ddx{w}{x_k} {x_j}  \Big\}\\
             &\quad  -(\tau \mathcal L(\Phi) -\Phi)|\nabla w |^2 -4\tau \sum_{i,j,k=1}^n a_{ij}\dx{\Phi}{x_i}\dx{w}{x_k}\ddx{w}{x_k}{x_j}  
               - 2N\sum_{i,j=1}^n a_{ij}\dx{w}{x_i}\dx{w}{x_j}\\
               &\quad -2\tau\Phi\sum_{k=1}^n \frac{\partial\fo}{\partial x_k}- 2Nw  \fo   -N_1 .
\end{aligned}
\eeq
We estimate the right-hand side of \eqref{Ltilw} term by term. 
Let $\varep>0$. The numbers $K_i$, for $i=1,2,3\ldots$, used in the calculations below are all positive and independent of $w$.
We denote the matrix of second derivatives of $w$ by $\nabla^2 w$, and denote its Euclidean norm by $|\nabla^2 w|$.
Note that $ \Am$, $\mathbf b$ and $\mathbf{\tilde b}$ are bounded in $B_\delta(\x^*)$.  This and Cauchy-Schwarz inequality  imply
\begin{align*}
& 2\tau\Phi  \sum_{i,j,k=1}^n \dx{a_{ij}}{x_k}\dx{w}{x_k} \ddx{w}{x_i}{x_j}  
\le 2 C \tau\Phi  |\nabla w| |\nabla^2 w|^2
\le \varepsilon^{-1} K_1 |\nabla w|^2   +2 \varepsilon\tau\Phi   |\nabla^2 w|^2,\\
& -(\tau \mathcal L(\Phi) -\Phi)|\nabla w |^2+2\tau\Phi\sum_{i,k=1}^n \dx{\tilde b_{i}}{x_k}\dx{w}{x_k} \dx{w}{x_i}\le K_2|\nabla w|^2.
\end{align*}
Since $\Am$ is positive definite, 
\beqs
 \sum_{i,j,k=1}^n a_{ij}\ddx{w}{x_k}{x_i}\ddx{w}{x_k} {x_j} \ge K_3 |\nabla^2 w|^2,
\quad 
\sum_{i,j=1}^n a_{ij}\dx{w}{x_i}\dx{w}{x_j}\ge K_3 |\nabla w|^2. 
\eeqs
Also, we have
\beqs 
-4\tau \sum_{i,j,k=1}^n a_{ij}\dx{\Phi}{x_i}\dx{w}{x_k}\ddx{w}{x_k}{x_j}  \le \varepsilon^{-1} K_4 |\nabla w|^2   +2\varepsilon \tau |\nabla \Phi|^2 |\nabla^2 w|^2,
\eeqs
\beqs
-2\tau\Phi\sum_{k=1}^n \frac{\partial \fo}{\partial x_k} \le K_5\deltafo ,
\eeqs
and by using estimate \eqref{main0} for $w$,
\beqs
 - 2Nw \fo \le K_6\deltao N \big[e^{-\eta_1 t_0}\sup_U |w(\x,0)|+\deltao\big].
\eeqs
 Combining the above estimates, we obtain from \eqref{Ltilw} that 
\begin{align*}
\mathcal L \tilde w
&\le 2\tau \Phi\Big(2 \varep +\varep\frac{|\nabla \Phi|^2}\Phi - K_3 \Big) |\nabla^2 w|^2 
   +  \Big( K_2 + \varepsilon^{-1}(K_1+K_4) - 2N K_3 \Big)|\nabla w |^2\\
&\quad + K_5\deltafo + K_6\deltao N \big[e^{-\eta_1 t_0}\sup_U |w(\x,0)|+\deltao\big]-N_1 .
\end{align*}
Since   $|\nabla \Phi|^2/\Phi\le 16\delta^2$, we have
\beq\label{Ltildew}
\begin{aligned}
\mathcal L \tilde w
&\le 2\tau \Phi\Big(K_7\varep - K_3 \Big) |\nabla^2 w|^2  
   +  \Big( K_2 + K_8\varepsilon^{-1}- 2N K_3 \Big)|\nabla w |^2\\
&\quad + (K_5 + K_6 N)\big[\deltafo+ \deltao e^{-\eta_1 t_0}\sup_U |w(\x,0)|+\deltao^2\big]-N_1 .
\end{aligned}
\eeq
In \eqref{Ltildew}, choose $\varep=K_3/K_7$ and $N=[K_2+K_8\varep^{-1}]/(2 K_3)$,
then take 
\beqs 
N_1=(K_5 + K_6 N)(\deltao e^{-\eta_1 t_0}\sup_U |w(\x,0)|+\deltao^2+\deltafo).
\eeqs 
We find that 
$ \mathcal L\tilde w \le0$ in $G_{\delta}$.
 Applying the maximum principle gives
\beq\label{Bdnwtil}
\max_{\bar G_{\delta}} \tilde w = \max \big \{ \tilde w(\x,t): (\x,t)\in  B_{\delta}(\x_*)\times\{t_0\}
 \cup \partial B_{\delta}(\x_*) \times[t_0,t_0+1] \big \}. 
\eeq
Note that $\tau \Phi(\x)=0$ when $t=t_0$ or $\x\in\partial B_{\delta}( \x_*)  $.
Hence \eqref{Bdnwtil} implies, 
 \beq\label{Bdnwti2}
\max_{\bar G_{\delta}} \tilde w  \le  N \max_{B_{\delta}( \x_*) } w^2(\x,t_0)+ N\max_{\partial B_{\delta}(\x_*) \times[t_0,t_0+1]} w^2(\x,t)
  +N_1.
 \eeq
Using estimate \eqref{main0} for the first two terms on the right-hand side of \eqref{Bdnwti2} we obtain   
\begin{align*}
\max_{\bar G_{\delta}} \tilde w 
&\le 2K_9 N \Big[e^{-\eta_1 t_0}\sup_U |w(\x,0)|+\deltao\Big]^2 +N_1
\le K_{10}\Big[e^{-2\eta_1 t_0}\sup_U |w(\x,0)|^2+\deltao^2+\deltafo\Big]\\
&\le C \Big[e^{-2\eta_1 t}\sup_U |w(\x,0)|^2+\deltao^2+\deltafo\Big].
\end{align*}
Now, we consider $\x \in B_{\delta/2}(\x_*)$. 
If $t\in(0,1]$ let $t_0=t/2$, then $t=2t_0\in[t_0,1+t_0]$ and hence  
     \begin{multline}\label{berngradout1}
 \frac t {2}|\nabla w (\x,t) |^2\min_{B_{\delta/2}(\x_*) }\Phi(\x)   \le (t-t_0)\Phi(\x) |\nabla w(\x,t)|^2 \\
 \le \tilde w(\x,t)
 \le C \Big[e^{-2\eta_1t}\sup_U |w(\x,0)|+\deltao^2 +\deltafo \Big] .
 \end{multline}
If $t>1$ let $t_0=t-1/2$, then $t\in [t_0,1+t_0]$ and hence
\begin{multline}\label{berngradout2}
 \frac 1 2|\nabla w (\x,t) |^2\min_{B_{\delta/2}(\x_*) }\Phi(\x)   \le (t-t_0)\Phi(\x) |\nabla w(\x,t)|^2 \\
 \le  \tilde w(\x,t)
 \le   C\Big[e^{-2\eta_1t}\sup_U |w(\x,0)|+\deltao^2 +\deltafo \Big].
\end{multline}
Since $\min_{B_{\delta/2}(\x_*) }\Phi(\x)>0$, it follows \eqref{berngradout1} and \eqref{berngradout2} that  
\beq\label{fordel1}
|\nabla w(\x,t)|\le M(\delta) \Big[ 1+\frac1{\sqrt t}\Big]
\Big[e^{-\eta_1 t}\sup_U |w(\x,0)|+\deltao +\sqrt{\deltafo} \Big]
\eeq
for $\x\in B_{\delta/2}(\x_*)$ and $t>0$.
Then using a finite covering of $U'$, we obtain  \eqref{Gradbound} from \eqref{fordel1}.
\end{proof}

We return to the IBVP \eqref{sigEq} for $\sigma(\x,t)$ now. Recall that the existence and uniqueness of the solution $\sigma$ were already addressed in Theorem \ref{eu}.


\begin{theorem}\label{Bsig}
Assume \condEa\ and  
 \beq\label{smalldata1}
\deltath\eqdef \sup_{D} (|\mathbf V(\x,t)| + |\nabla\mathbf V(\x,t)|) +\sup_{\Gamma\times [0,\infty)} |g(\x,t)| <\infty.
\eeq 
Then the solution $\sigma(\x,t)$ of the IBVP \eqref{sigEq} satisfies 
\beq\label{Boundsig}
\sup_{\x\in U} |\sigma(\x,t)| \le C\Big[ e^{-\eta_1t}\sup_U |\sigma_0(\x)|+\deltath\Big] \quad \text{for all } t>0 .
\eeq
Moreover,  
\beq\label{sigsmall}
\limsup_{t\to\infty} \left[\sup_{\x\in U}|\sigma(\x,t)|\right]  \le C\deltathp,
\eeq 
where
\beq
\deltathp=\limsup_{t\to\infty}\Big[ \sup_{\x\in U} (|\mathbf V(\x,t)| + |\nabla\mathbf V(\x,t)|) +\sup_{\x\in \Gamma} |g(\x,t)|\Big].
\eeq
\end{theorem}
\begin{proof}
Let $w(\x,t) = \sigma(\x,t) e^{-\Lambda(\x) }$, $\fo(\x,t)=e^{-\Lambda(\x)}\nabla\cdot(\Am(\x)\mathbf c(\x,t))$, $G(\x,t)=e^{-\Lambda(\x) }g(\x,t)$ and $w_0(\x)=e^{-\Lambda(\x)}\sigma_0(\x)$.
Then $w(\x,t)$ solves \eqref{nonhom}.
We observe from \eqref{Lamest1} that
\beq
|\Lambda(\x)|\le d_2\mu_3(R-r_0)\quad \forall \x\in \mscrU.
\eeq
Combining with the boundedness of $\|\Am\|_{C^1(\mscrU)}$, we have
\beq
|\fo(\x,t)|\le C(|\mathbf V(\x,t)|+|\nabla \mathbf V(\x,t)|)\quad \forall (\x,t)\in D.
\eeq
Thanks to these relations, the assumptions in Proposition~\ref{Allt} hold, thus, the assertions \eqref{Boundsig} and \eqref{sigsmall}  follow directly from \eqref{main0} and \eqref{limsup0}.        
\end{proof}

%


For the velocities,  we have the following result. 


\begin{theorem}\label{mainv}
Assume \condEa\ and 
 \beq\label{smalldata2}
 \deltafi \eqdef \sup_{D} (|\mathbf V(\x,t)| + |\nabla\mathbf V(\x,t)| +|\nabla^2\mathbf V(\x,t)|)<\infty
\text{ and }
\deltasi \eqdef \sup_{\Gamma\times [0,\infty)} |g(\x,t)|< \infty.
\eeq 
Then for any $U'\Subset U$, there is a positive number $\tilde M$ such that
for $i=1,2$, and $t>0$,
\beq\label{supv12} 
  \sup_{ \x\in U' } |\V_i(\x,t)|\le  \tilde M \Big(1+\frac 1{\sqrt t}\Big)\Big[e^{-\eta_1t}\sup_U |\sigma_0(\x)|+\deltafi +\sqrt \deltafi + \deltasi \Big].
 \eeq
Consequently, if 
 \beq
\lim_{t\to\infty}\Big \{ \sup_{\x\in U} (|\mathbf V(\x,t)| + |\nabla\mathbf V(\x,t)| +|\nabla^2\mathbf V(\x,t)|) +\sup_{\x\in\Gamma} |g(\x,t)|\Big\} =0, 
\eeq
then for any $\x\in U$,
\beq\label{limvelo}
\lim_{t\to\infty} \V_1(\x,t)=\lim_{t\to\infty} \V_2(\x,t)=0.
\eeq 
\end{theorem}
\begin{proof} 
Note that solution $\sigma(\x,t)$ of \eqref{sigEq} satisfies $\sigma\in C_{\x}^3(D)$ and $\sigma_t\in C_{\x}^2(D)$.
Let $w,\fo,G,w_0$ be the same as in Theorem \ref{Bsig}.
Using the estimate of $\nabla w$ in Lemma \ref{gradw} and formula \eqref{v2w}, we easily obtain estimate \eqref{supv12} for $\V_2$. Then the estimate for $\V_1$ follows this and \eqref{v1}. The proof of \eqref{limvelo} is similar to that of \eqref{limsup0}. We take $U'=B_{\delta}(\x)$ such that $U'\Subset U$. 
For $T>0$, let
 \beqs
 \deltafiT = \sup_{U\times[T,\infty)} (|\mathbf V(\x,t)| + |\nabla\mathbf V(\x,t)| +|\nabla^2\mathbf V(\x,t)|)
\text{ and }
\deltasiT = \sup_{\Gamma\times [T,\infty)} |g(\x,t)|.
\eeqs 
Use \eqref{supv12} for all $t>T$ and $\deltafiT$, $\deltasiT$ in place of $\deltafi$, $\deltasi$. Then let $T\to\infty$ noting that $\deltafiT\to 0$ and $\deltasiT\to 0$.
\end{proof}

\begin{remark}
The key ingredient of the above asymptotic results is Lemma \ref{grow}, the lemma of growth in time. It is worth mentioning that this result can be extended to more general parabolic equations in more general domains $D$ in $\R^{n+1}$ rather than just cylindrical-in-time domains $D=U\times (0,\infty)$.
\end{remark}

\myclearpage
\section{Case of unbounded domain}\label{unbounded}

We will analyze the linear stability of the steady flows from section \ref{Steady} in an unbounded, outer domain $U= \R^n\setminus\bar \Omega$,
where   $\Omega$ is a simply connected, open, bounded set containing the origin.
To emphasize the ideas and techniques, we consider the simple case $\Omega=B_{r_0}$ for some $r_0>0$.


For $R>r>0$, denote  
$\mscrO_{r}=\R^n\setminus \bar B_r$, 
$\mscrO_{r,R}=B_R \setminus \bar B_r$,
and denote their closures by $\bar \mscrO_r$ and $\bar \mscrO_{r,R}$, respectively.
Then $U=\mscrO_{r_0}$. Let $\Gamma=\partial U=\{\x:|\x|=r_0\}$ and $D=U\times(0,\infty)$.


For $T>0$ we denote $U_T=U\times(0,T]$, then its closure is $\bar U_T=\bar U\times[0,T]$ and its parabolic boundary is 
$\partial_p U_T=[\bar U\times\{0\}]\cup [\Gamma\times(0,T]]$.

Same as in section \ref{bounded}, we consider a steady state $(u^*_1(\x),u^*_2(\x),S_*(\x))$ in \eqref{Sradial} with $c_1^2+c_2^2>0$ and 
$\hat S(r)$ exists for all $r\ge r_0$. 
We assume throughout this section that
\beq\label{hatScond}
0<\underline s \le \hat S(r)\le \bar s<1 \quad\forall r\ge r_0,\text{ where }\underline s,\bar s=const.
\eeq
For instance,  in one of the cases in Theorem \ref{globalODE} if the limit $s_\infty\eqdef \lim_{r\to\infty}\hat S(r)$, which exists according to Theorem \ref{genS}, belongs to the interval $(0,1)$ then \eqref{hatScond} holds.

The problems of our interest are \eqref{sigEq} and its transformed form \eqref{nonhom}. 

Let $\mu_i$, for $i=1,2,3$, and $C_j$, for $j=0,1,2$, be defined as in section \ref{bounded} (see \eqref{mudef}, \eqref{Aellip1} and \eqref{Abbound1}).
Thanks to condition \eqref{hatScond}, which plays the role of \eqref{Sassump} in section \ref{bounded}, the main properties \eqref{Aellip1}, \eqref{Abbound1} and \eqref{smoooth1} still hold with $\mscrU=\mscrO_{r_0,R}$ being replaced by $\mscrU=U=\mscrO_{r_0}$.


\subsection{Maximum principle for unbounded domain} 
We establish the maximum principle for equation $\mathcal Lw=0$  in the domain $U$ with operator $\mathcal L$ defined by \eqref{Lw}.
For $T>0$, we construct a barrier function $W(\x,t)$  of the form: 
\beq\label{sup-sol}
 W(\x,t)\eqdef(T-t)^{-s} e^\frac{\varphi (\x)}{T-t}  \quad \text{for } (\x,t) \in \mscrO_{r_0,R}\times (0,T),
\eeq 
where constant $s>0$ and function $\varphi(\x)>0$ will be decided later.
Elementary calculations give
\beqs
 \mathcal L W = (T-t)^{-s-2}e^\frac{\varphi}{T-t} \Big\{ (T-t)\big(s-\nabla\cdot(\Am\nabla \varphi)- {\bf b}\cdot \Am\nabla \varphi  \big)  +\varphi - (\Am\nabla \varphi) \cdot \nabla \varphi   \Big\}.
 \eeqs
Then $\mathcal L W\ge 0$ if  
\beq\label{Setcond}
 s\ge \nabla\cdot(A\nabla \varphi)+ {\bf b}\cdot \Am \nabla \varphi \quad \text{and}\quad
 \varphi \ge (\Am\nabla \varphi) \cdot \nabla \varphi. 
\eeq
Similar to section \ref{bounded}, we choose
\beq\label{Phidef}
\varphi(\x) =\kapo\Big(\varphi_1 +\int_{r_0}^{|\x|} r \phi(r) dr\Big), \quad \text{where } \varphi_1=\frac {C_1r_0^2}2>0\text{ and }  \kapo=\frac{C_1}{2C_0},
\eeq
and function $\phi$ is defined by \eqref{defphi}. 
As in Lemma~\ref{WsubPar}, we have 
\beq\label{eqvar}
\Am\nabla \varphi =\kapo \x\quad \text{and}\quad\nabla \varphi =\kapo\phi(|\x|)\x .
\eeq
By \eqref{phiest0}, $\phi(r) \ge  d_0 \mu_2 = C_1>0$.
Then
\beq\label{phi2}
\varphi(\x) \ge \kapo \Big(\varphi_1+ C_1\int_{r_0}^{|\x|} r dr\Big)  =\frac  {\kapo C_1} 2 |\x|^2.   
\eeq
Also, we see from \eqref{eqvar} and \eqref{B3} that 
\beq\label{p2}
(\Am\nabla \varphi) \cdot \nabla \varphi = \kapo^2 \x^T \Bm \x 
\le  d_4  \kapo^2|\x|^2 \sum_{i=1}^2 F_i ( S_*(\x))\le \kapo^2 d_4 \mu_1 |\x|^2 = \kapo^2C_0|\x|^2=\frac  {\kapo C_1} 2 |\x|^2.
\eeq
Then we have from \eqref{phi2} and \eqref{p2} that
\beq\label{ksmall}
  \varphi(\x) \ge   (A\nabla \varphi) \cdot \nabla \varphi . 
\eeq
By \eqref{Abbound1} and \eqref{eqvar}, we have 
\beqs 
 \nabla\cdot(\Am\nabla \varphi)+ {\bf b}\cdot \Am\nabla \varphi \le \kapo(n+ C_2|\x|)\le C_3(1+|\x|),\quad\text{where } C_3=\kapo(n+C_2).
\eeqs
Select 
\beq\label{slarge}
s= s_R\eqdef C_3(1+R),
\eeq  
then 
\beq\label{news}
 s\ge \nabla\cdot(\Am\nabla \varphi)+ {\bf b}\cdot \Am\nabla \varphi \quad \text{in } \mscrO_{r_0,R}.
\eeq
Therefore  $\mathcal L W \ge 0$ in $\mscrO_{r_0,R}\times (0,T)$.
We summarize the above arguments in the following lemma.

 \begin{lemma}\label{sup-par}
Let $T>0$, $R>r_0$ and let the function  $\varphi$ be defined by \eqref{Phidef}. Then for $s=s_R$ in \eqref{slarge}, the function $W(\x,t)$ in \eqref{sup-sol} belongs to $C_{\x,t}^{2,1}(D)\cap C(\bar D)$  and satisfies $\mathcal L W\ge 0$  in   $\mscrO_{r_0,R}\times (0,T)$. 
 \end{lemma} 

 Using the above  barrier function $W(x,t)$, we have the following maximum principle.  

  \begin{theorem}\label{maximumprinciple} 
Let  $T>0$ and $w(\x,t)$ be a bounded function in $C^{2,1}_{\x,t}(U_T) \cap C(\bar U_T)$ that solves $\mathcal L w = \fo$ in $U_T$, where $\fo\in C(\bar U_T)$. 
Then
\beq\label{maxunbound}
\sup_{\bar U_T}   |w(\x ,t)|  \le \sup_{\partial_p U_T} |w(\x ,t)| + (T+1)\sup_{\bar U_T} |\fo|.
\eeq
\end{theorem}
\begin{proof}
Given any $(\x_0,t_0)\in U\times (0,T)$.
Let $\delta>0$ such that $t_0<T-\delta$.
Let $M=\sup_{\bar U_T}|w(\x,t)|$ and $\deltaz=\sup_{\bar U_T} |\fo|$ which are finite numbers. Let $\mu>0$ be arbitrary.
Select $R>0$ sufficiently large such that
\beq\label{Rcond} 
T^{-C_3(1+R)} e^\frac{\kapo C_1 R^2}{2T}>M/\mu.
\eeq
Denote $\mathcal C= \mscrO_{r_0,R}\times(0,T-\delta]$. Then $(\x_0,t_0)\in \mathcal C$.
Let $W(\x,t)$ be as in Lemma \ref{sup-par}.
We define the auxiliary function  
\beq\label{u:def}
u(\x,t)=w(\x,t)-\deltaz(t+1) -\mu W(\x, t),\quad (\x,t)\in \mathcal C.
\eeq   
We have  $u\in C^{2,1}_{\x,t}(\mathcal C) \cap C( \mathcal C)$ and, thanks to Lemma \ref{sup-par}, function $u$ satisfies
$$  \mathcal Lu=   \fo -\deltaz  -\mu \mathcal  L  W \le 0 \quad\text{in } \mathcal C. $$
By the maximum principle,
\beq\label{Laxprinciple}
\max_{\bar {\mathcal C} }u  = \max_{\partial_p \mathcal C}   u.
\eeq 
Let us evaluate $u(x,t)$ on the parabolic boundary $\partial_p \mathcal C$.
For any $\x\in  \mscrO_{r_0,R}$,
\beq\label{LT0}
u(\x,0) \le  w(\x,0)-\mu W(\x,0)=w(\x,0)-  \mu T^{-s} e^\frac{\varphi (\x)}{T} \le    w(\x,0).
\eeq
For $|\x|=r_0$ and $0\le t\le T-\delta$,
\beq\label{LT1}
u(\x,t) \le  w(\x ,t) -\mu W(\x,t) \le w(\x ,t).
\eeq
For $|\x|=R$ and $0\le t\le T-\delta$, we have from \eqref{phi2}, \eqref{slarge} and \eqref{Rcond} that
\beq\label{LT2}
u(\x,t) \le   w(\x,t) - \mu(T-t)^{-s} e^\frac{\varphi (\x)}{T-t}  \le  M-\mu T^{-C_3(1+R)} e^\frac{\kapo C_1 R^2}{2T} \le 0.
\eeq
Hence, we have from \eqref{Laxprinciple}, \eqref{LT0},\eqref{LT1} and \eqref{LT2} that
\beq\label{LaxonBoundary}
 \max_{\bar{ \mathcal C}} u(\x,t) 
 \le \max\{ 0, \sup_{U} w(\x,0), \sup_{\Gamma\times[0,T]} w(\x ,t)   \}. 
\eeq                             
In particular, it follows from \eqref{LaxonBoundary} that 
\beq\label{u:max}
  u(\x_0,t_0)\le  \max\{0,\sup_{\partial_p U_T} w\}.
\eeq
Now, letting $\mu\to 0$ in \eqref{u:def} yields
\beqs
   w(\x_0,t_0)-\deltaz(t_0+1)\le  \max\{0,\sup_{\partial_p U_T}  w\} \le  \sup_{\partial_p U_T}  |w|  .
\eeqs
Hence,
\beqs
   w(\x_0,t_0)\le  \sup_{\partial_p U_T}  |w| + \deltaz(T+1).
\eeqs
Repeating the above arguments for $(-w)$ gives
\beq
   |w(\x_0,t_0)|\le   \sup_{\partial_p U_T}  |w| + \deltaz(T+1)
\eeq
for any $(\x_0,t_0)\in U\times (0,T)$. Therefore, \eqref{maxunbound} follows.
\end{proof}

We study the following IBVP \eqref{nonhom} for $w(\x,t)$.

\textbf{Condition} \condEb. $F_1,F_2\in C^7((0,1))$, $w_0\in C(\bar U)$, $G\in C(\Gamma\times[0,\infty))$ and $G(\x,0)=w_0(\x)$ on $\Gamma$.

\begin{theorem}\label{Euniq} 

Assume \condEb, $f_0\in C_{\x}^5(\bar D)$, $\partial_t f_0\in C_{\x}^3(\bar D)$.
Suppose $w_0(\x)$, $G(\x,t)$ and $\fo(\x,t)$ are bounded functions.
Then, 

{\rm (i)} There exists a solution $w(\x,t)\in C^{2,1}_{\x,t}(D)\cap C(\bar D)$ of  \eqref{nonhom} . 

{\rm (ii)} This solution is unique in class of locally (in time) bounded solutions, i.e., the class of solutions $w(\x,t)$ such that
\beq
\sup_{ U\times [0,T]} |w(\x,t)|<\infty\quad \text{for any } T>0.
\eeq

{\rm (iii)} Furthermore,  for $(\x,t)\in D$, 
\beq\label{westimate}
|w(\x,t)|\le  \deltaten + \deltazero(t+1),
\eeq
where
\beq
\deltaten = \max\{\sup_{ U }| w_0(\x)|,\sup_{\Gamma\times[0,\infty)}|G(\x,t)|\}\quad\text{and}\quad
\deltazero=\sup_{D} |\fo|.
\eeq
\end{theorem}
\begin{proof}
We rewrite equation in the non-divergent form   
\beqs
\mathcal Lw= \frac{\partial w} {\partial t}- \sum_{i,j=1}^n a_{ij}\frac{\partial^2 w} {\partial x_i\partial x_j}  - \sum_{i,j=1}^n \big[ (a_{ij})_{x_i} +a_{ij}b_i \big]  \frac{\partial w}{\partial x_j} = 0.
\eeqs
Thanks to Theorem 4 p.474 of \cite{IKO} and the maximum principle in Theorem \ref{maximumprinciple}, one can prove (i), (ii) and (iii) using similar  arguments presented in Theorem 4.6 of \cite{HIK1}. We omit the details.
\end{proof}

\subsection{Lemma of growth in spatial variables} 

We now study the behavior of the solutions as $|\x|\to\infty$. This requires a different type of lemma of growth and a new barrier function.

Let  $R>0$ and $\ell\ge R+r_0$. Denote
\beq
\mscrO_R(\ell)=\mscrO_{\ell-R,\ell+R}=\{\x\in\R^n: | |\x|  -\ell |< R\} \quad \text{and}\quad \Sph_\ell=\{\x\in\R^n:|\x|=\ell\}.
\eeq
Define the barrier function
\beq\label{sub-sol}
\mathcal W(\x,t) =\frac1{(t+1)^s} e^{-\frac{\psi(\x)}{t+1} }\quad\text{for } |\x|\ge r_0,\ t\ge 0,
\eeq
where parameter $s>0$ and function $\psi>0$.
Then
\beq
\mathcal L \mathcal W=(t+1)^{-s-2} e^{-\frac{\psi(\x)}{t+1} }\Big\{ (t+1)\big[-s+\nabla\cdot(\Am\nabla \psi)+{\bf b}\cdot \Am \nabla \psi\big]+\psi-(\Am\nabla \psi) \cdot \nabla \psi \Big\}.
\eeq
Hence,  $\mathcal L\mathcal W\le 0$ if 
 \beq\label{spsicond}
 s\ge \nabla\cdot(\Am\nabla \psi)+ {\bf b}\cdot \Am\nabla \psi\quad\text{and}\quad
 \psi \le (\Am\nabla \psi) \cdot \nabla \psi .
 \eeq
 Denote ${\boldsymbol \xi}(\x) =\ell\x/|\x|$.
We will choose $\psi$ such  that 
\beqs
\Am\nabla \psi = \kaptw(\x-\boldsymbol\xi)\quad\text{for some }\kaptw>0.
\eeqs
By \eqref{Bxx} and \eqref{defphi},
\beq
\nabla \psi =\kaptw \Am^{-1}(\x-\boldsymbol\xi) =\kaptw \Bm \x (|\x|-\ell)/|\x|
=\kaptw\phi(|\x|) (|\x|-\ell) \x/|\x|.  
\eeq
Select 
\beq\label{newphi}
\psi(\x) = \kaptw\int_{\ell}^{|\x|} (r-\ell)\phi(r) dr,\quad   \text{where }     
\kaptw= \frac{C_0}{2C_1 }
\eeq 
and function $\phi$ is defined by \eqref{defphi}.
For all $\x\in \mscrO_R(\ell)$, we have  from \eqref{phiest1} that
\beq\label{psibound}
\psi ( \x)\le \kaptw C_0   \int_{\ell}^{|\x|} (r -\ell) dr
=  \frac{\kaptw C_0} 2 ( |\x| - \ell )^2.
\eeq
By \eqref{B3}, 
\begin{align*}
(\Am\nabla \psi) \cdot \nabla \psi &=\kaptw^2 (\x-\boldsymbol\xi)^T \Bm(\x) (\x-\boldsymbol\xi)
 \ge d_0 \kaptw^2 |\x-\boldsymbol\xi|^2\sum_{j=1}^2 F_j(S_*(\x))\\
&\ge \kaptw^2 C_1(|\x|-\ell)^2
=  \frac{\kaptw C_0} 2 ( |\x| - \ell )^2.
\end{align*}
Hence  this and \eqref{psibound}  give $\psi \le (\Am\nabla \psi) \cdot \nabla \psi$, that is, the second condition  in \eqref{spsicond}.
Also,
\begin{align*}
 \nabla\cdot(\Am\nabla \psi)+ {\bf b}\cdot (\Am\nabla \psi) 
 &=\kaptw\Big[\nabla\cdot (\x -\boldsymbol\xi)+ {\bf b}\cdot (\x-\boldsymbol\xi) \Big]
 =\kaptw\Big[n-(n-1)\frac{\ell}{|\x|} + {\bf b}\cdot(\x-\boldsymbol\xi) \Big]\\
 &\le \kaptw (n + |{\bf b}| R).
\end{align*}
Then by \eqref{Abbound1},
\beq
 \nabla\cdot(\Am\nabla \psi)+ {\bf b}\cdot (\Am\nabla \psi)  \le \kaptw(n + C_2 R)\le  C_3(1+R) ,   
\eeq
where $C_3=\kaptw (n+C_2)$.
By selecting
\beq\label{s-pow}
s=s_R\eqdef  C_3(1+R),  
\eeq
we have $s\ge \nabla\cdot(\Am\nabla \psi)+ {\bf b}\cdot (\Am\nabla \psi)$ which is the first condition in \eqref{spsicond}.
Therefore  $\mathcal LW\le 0$ in $\mscrO_R(\ell)\times(0,\infty)$. We have proved:

\begin{lemma}\label{Sub-sol}
Given any $R>0$ and $\ell\ge R+r_0$.
Let $s=s_R$ be defined by \eqref{s-pow}  and the function $\psi$ be defined by \eqref{newphi}. Then the function $\mathcal W(\x,t)$ in \eqref{sub-sol} belongs to $C_{\x,t}^{2,1}(D)\cap C(\bar D)$ and satisfies $\mathcal L\mathcal W\le 0$ on $\mscrO_R(\ell)\times(0,\infty)$. 
 \end{lemma}

 Next is the lemma of growth in the spatial variables. 

\begin{lemma}\label{Cmp}
Given $T>0$, let
\begin{align}
\label{RT}
 R&=R(T)=C_4(1+T),\\
\label{etaT}
\eta_0&=\eta_0(T)=\Big(1 -\frac 1 {2^{C_5(T+1)}}\Big)\frac 1{(T+1)^{2C_5(T+1)}},
\end{align}
where $C_4= \max\{1, \frac{8C_3}{\kaptw eC_0}\}$ and $C_5=C_3C_4$.
Suppose $w(\x,t)\in C_{\x,t}^{2,1}(U_T)\cap C(\bar U_T)$ satisfies   $\mathcal L w \le 0$  on  $U_T$  and  $w(\x,0) \le 0$ on $\bar U$.
Let $\ell$ be any number such that  $\ell \ge R+r_0$, then 
 \beq\label{Comp}
 \max\big\{0, \sup_{ \Sph_\ell\times [0,T] }  w(\x, t) \big\}\le  \frac1 {1+\eta_0 } \max \big\{ 0, \sup_{ \bar \mscrO_R(\ell)\times[0,T] }  w(\x, t)\big \}. 
\eeq
\end{lemma}
\begin{proof}
Denote
\begin{align*}
M_\ell &=\max\big\{ 0, \sup_{ \bar \mscrO_R(\ell)\times[0,T] }  w(\x, t)|\big\}  \quad \text{and}\quad
m_\ell =\max\big\{ 0,\sup_{ \Sph_\ell\times [0,T] }  w(\x, t)\big\}.
\end{align*}
Let $\mathcal W$ be defined as in Lemma \ref{Sub-sol}.
Let $\eta>0 $ chosen later and define 
\beqs
\widetilde W(\x,t) = M_\ell ( 1- \mathcal W(\x,t) +\eta),
\eeqs
then $ \mathcal L \widetilde {\mathcal W}\ge 0 $ in $\mscrO_R(\ell)\times (0,T]$. 
We have  
\beq\label{b1}
\widetilde W(\x,0)  =M_\ell(1- e^{-\psi(\x) }+\eta )\ge 0\ge w(\x,0).
\eeq
By \eqref{psibound}, $ \psi(\x) \le \kaptw C_0 R^2/2 $ when $|\x|=\ell\pm R$, hence
\beq\label{wt2}
\widetilde W(\x,t)|_{|\x|=\ell\pm R} \ge M_\ell \Big( 1- (t+1)^{-s} e^{-\frac {\kaptw C_0 R^2}{2(t+1) }}+\eta  \Big).
\eeq
Let $f(z)=z^{-s} e^{-\frac {\kaptw C_0 R^2}{2z}} $ for $z\ge 0$. 
Select $\eta=\max_{[0,\infty)}f(z)$.
Elementary calculations show 
$\eta = (\frac{2 s}{\kaptw eC_0 R^2} )^s $.
Then $t\in[0,T]$, it follows \eqref{wt2} that
\beq\label{b2} 
\widetilde W(\x,t)|_{ |\x| =\ell\pm R}  \ge  M_\ell\ge \max\{0, w(\x, t)|_{|\x| =\ell\pm R} \}.
\eeq
From \eqref{b1}, \eqref{b2} and maximum principle we obtain 
\beqs
\widetilde W(\x,t) \ge w(\x,t) \quad \text {on }  \bar \mscrO_R(\ell)\times(0,T).
\eeqs
Particularly, 
\beq\label{Wk0}
\widetilde W(\x,t)\ge w(\x,t)\quad \text {on }  \Sph_\ell\times(0,T).
\eeq
Moreover, since $\psi (\x)=0$ when $|\x|=\ell$, $\mathcal W(\x,t)\ge \frac1{(T+1)^s}$ thus 
\beq\label{Wk1}
\widetilde W(\x,t)|_{|\x|=\ell}  \le  M_\ell\Big[ 1- \frac 1 {(T+1)^s} +\eta\Big] .
\eeq 
Since $R\ge 1$, we easily estimate
\beqs
\eta = \Big[\frac{2C_3(1+R)}{\kaptw  e C_0 R^2} \Big]^{C_3(1+R)} \le \Big(\frac{4C_3 R}{\kaptw  e C_0 R^2} \Big)^{C_3(1+R)} \le  \Big(\frac{C_4}{2R}\Big)^{C_3(1+R)}. 
\eeqs
Hence
\beq\label{Tseta0}
\begin{aligned}
\frac 1 {(T+1)^s} -\eta &\ge \frac 1 {(T+1)^{C_3(1+R)}} - \Big(\frac{C_4}{2R}\Big)^{C_3(1+R)}
= \Big(1 -\frac 1 {2^{C_3(R+1)}}\Big)\frac 1{(T+1)^{C_3(1+R)}}\\
  &\ge \Big(1 -\frac 1 {2^{C_3R}}\Big)\frac 1{(T+1)^{2C_3R}}=\Big(1 -\frac 1 {2^{C_5(T+1)}}\Big)\frac 1{(T+1)^{2C_5(T+1)}}=\eta_0.
\end{aligned}
\eeq
From \eqref{Wk0}, \eqref{Wk1} and \eqref{Tseta0} we obtain 
$(1-\eta_0)M_\ell\ge m_\ell$,  thus, $M_\ell \ge \frac{m_\ell}{1-\eta_0}   \ge (1+\eta_0)m_\ell$,
which gives \eqref{Comp}. 
\end{proof}

\begin{lemma} \label{cyls}
Let $T>0$ and $R$, $\eta_0$ and $w(\x,t)$ be as in Lemma \ref{Cmp}. For $i\ge 1$, let 
 \beq\label{mi0}
\bar m_i = \max\big\{0, \sup_{ \Sph_{r_0+iR}\times[0,T] }  w(\x, t) \big\}.
\eeq

Part A (Dichotomy for one cylinder). Then for any $i\ge 1$, we have either of the following cases.
 \begin{itemize}
  \item[\rm (a)] If $\bar m_{i+1}\ge \bar m_{i-1}$, then $\bar m_{i+1}\ge (1+\eta_0)\bar m_i$. 
  \item[\rm (b)] If $\bar m_{i-1}\ge \bar m_{i+1}$, then $\bar m_{i-1}\ge (1+\eta_0)\bar m_i$.
 \end{itemize}

\medskip
Part B (Dichotomy for many cylinders). For any $k\ge 0$, we have the following two possibilities:
\begin{itemize}
 \item [\rm (i)] There is $i_0\ge k+1$ such that $\bar m_{i_0+j}\ge (1+\eta_0)^j\bar m_{i_0}$ for all $j\ge 0$.
 \item [\rm (ii)] For all $j\ge 0$, $\bar m_{k+j}\le (1+\eta_0)^{-j}\bar m_k$.
\end{itemize}
\end{lemma}
\begin{proof}
Part A.  By maximum principle,  
\begin{align*}
\sup_{ \bar \mscrO_R(r_0+iR)\times [0,T] }  w(\x, t) 
&\le \max\big\{ \sup_{\Sph_{r_0+(i\pm 1)R}\times [0,T]}  w(\x, t),  \sup_{\bar \mscrO_R(r_0+iR)}  w(\x, 0)  \big\}\\
&\le \max\big\{ \sup_{\Sph_{r_0+(i\pm 1)R}\times [0,T]}  w(\x, t),  0  \big\}
\le \max\{\bar m_{i-1}, \bar m_{i+1}  \}.  
\end{align*}
Hence,
\beq\label{supO}
\sup_{ \bar \mscrO_R(r_0+iR)\times [0,T] }  w(\x, t) 
\le \max\{\bar m_{i-1}, \bar m_{i+1}  \}.  
\eeq
Let $\ell=r_0+iR$.  Applying  Lemma ~\ref{Cmp} and \eqref{supO}, we obtain 
\beqs
\bar m_i \le \frac1 {1+\eta_0 } \max \big\{ 0, \sup_{ \bar \mscrO_R(r_0+iR)\times [0,T] }  w(\x, t) \big\} \le \frac{1}{1+\eta_0} \max\{\bar m_{i-1}, \bar m_{i+1}  \}.
\eeqs
Then the statements (a) and (b) obviously follow. 

\medskip
Part B. For $i<j$, define the cylinder 
\beqs
\mathcal C _{i,j}= \mscrO_{r_0+iR,r_0+jR}\times (0,T) =\{ (\x,t): r_0+iR < |\x| <  r_0+jR ,\ t\in (0,T) \}.
\eeqs
We say that (a) and (b) above are two cases for cylinder $\mathcal C_{i-1,i+1}$.

Let $k\ge 0$.  By Part A, we have either of the following two cases.
 
{\bf Case 1.} There is $i_0\ge k$ such that Case (a) holds for $\mathcal  C_{i_0, i_0+2}$, that is,
 \beq\label{a}  
  \bar m_{i_0+2}\ge \bar m_{i_0} \quad \text { and } \quad \bar m_{i_0+2}\ge (1+\eta_0)\bar m_{i_0+1}.
 \eeq
Then applying Part A to $\mathcal C _{i_0+1,i_0+3}$ we have either 
  \beq\label{aa}
  \text{Case (a) holds for $\mathcal C _{i_0+1,i_0+3}$, which gives } \bar m_{i_0+3} \ge   \bar m_{i_0+1} \text { and } \bar m_{i_0+3}\ge  (1+\eta_0) \bar m_{i_0+2},
  \eeq
or
\beq\label{ab} 
  \text{Case (b) holds for $\mathcal C _{i_0+1,i_0+3}$, which gives } \bar m_{i_0+1} \ge   \bar m_{i_0+3} \text { and } \bar m_{i_0+1}\ge (1+\eta_0)\bar m_{i_0+2}.
\eeq

Observe that \eqref{a} and \eqref{ab} hold simultaneously if only if  
\beq\label{seq0}
\bar m_{i_0}=\bar m_{i_0+1}=\bar m_{i_0+2}=\bar m_{i_0+3}=0,
\eeq 
which is a special case of \eqref{aa}. Hence we always have Case (a) for the next cylinder $\mathcal C _{i_0+1,i_0+3}$.
Then by induction, Case (a) holds for the cylinders $\mathcal C _{i_0+j-1,i_0+j+1}$ for all $j\ge 1$.
Thus,
\beq\label{ind}
\bar m_{i_0+j+1} \ge (1+\eta_0)\bar m_{i_0+j} \ge  (1+\eta_0)^2\bar m_{i_0+j-1} \ge \ldots\ge (1+\eta_0)^{j}\bar m_{i_0+1}.
\eeq 
Re-indexing $i_0+1$ by $i_0$ in \eqref{ind}, we obtain (i).   

{\bf Case 2.} For all $i\ge k$,   Case (b) holds for $\mathcal  C_{i, i+2}$, that is,
 $ \bar m_i \ge (1+\eta_0)\bar m_{i+1}$ for all $i\ge k$.
Therefore, 
  \beq \label{bb}
\bar m_k \ge (1+\eta_0)\bar m_{k+1} \ge (1+\eta_0)^2 \bar m_{k+2}  \ge \ldots\ge (1+\eta_0)^j \bar m_{k+j},
  \eeq 
 which implies (ii). 
\end{proof}

Using the above dichotomy, we obtain the behavior of a sub-solution $w$ as $|\x|\to\infty$.

\begin{proposition}\label{small0}  
 Assume $w\in C_{\x,t}^{2,1}(U_T)\cap C(\bar U_T)$ satisfies 
$w(\x,0)\le 0$ in $U$, $\mathcal Lw\le 0$ on $U_T$, and 
$w(\x,t)$ is bounded on $\bar U_T$.   
Then
\beq\label{assym0}
\limsup_{r\to\infty} (\sup_{\Sph_r\times[0,T]}w(\x,t))\le 0.
\eeq
\end{proposition}
\begin{proof} 
Let $\bar m_i$ be defined as in Lemma \ref{cyls}.

{\bf Case 1}: There are infinitely many $i$ such that $\bar m_i = 0$. Then there is a sequence $\{i_l\}$ increasing to  $\infty$ as $l\to \infty$ such that $\bar m_{i_l} =0$ for all $l\ge 1$. 
Then by maximum principle for cylinder $\mathcal C_{i_l,i_{l+1}}$ we have $w(\x,t)\le 0$ on $\mathcal C_{i_l,i_{l+1}}$ for all $l\ge 1$.
Therefore $w(\x, t) \le 0$ in $\{|\x|\ge r_0 +i_1 R\}\times [0,T]$.
This gives \eqref{assym0}. 

{\bf Case 2}: There are only finitely many $i$ such that $\bar m_i = 0$. Then there is $N>0$ such that $\bar m_i > 0$ for all $i\ge N$. 
We apply part B of Lemma \ref{cyls} to $k=N$.
If (i) holds, then there is $i_0\ge N+1$ such that $\bar m_{i_0+j}\ge (1+\eta_0)^j\bar m_{i_0}>0$ for all $j\ge 0$; thus, $\lim_{j\to\infty} \bar m_{i_0+j}=\infty$ which contradicts $w(\x,t)$ being bounded on $U_T$.  Hence we must have (ii), that is, for all $j\ge 0$, $\bar m_{N+j}\le (1+\eta_0)^{-j}\bar m_N$. Therefore, $ \lim_{j\to\infty} \bar m_{N+j} =0$ which, in combining with \eqref{supO}, proves \eqref{assym0}.
\end{proof}

As for solutions of the IBVP \eqref{nonhom} in a finite time interval, we have the following.

\begin{theorem}\label{small1} 
Let $w\in C_{\x,t}^{2,1}(U_T)\cap C(\bar U_T)$  be a bounded solution of \eqref{nonhom} on $U_T$ with $\fo\in C(\bar U_T)$. 
If
\beq\label{w0lim}
\lim_{|\x|\to\infty} w_0(\x)=0,
\eeq
\beq\label{focond}
\lim_{|\x|\to\infty} \sup_{0\le t\le T} |\fo(\x,t))| =0, 
\eeq    
then
\beq\label{assym3}
\lim_{r\to\infty} \Big(\sup_{\Sph_r\times[0,T]}| w(\x,t) |\Big)=0.
\eeq
\end{theorem}
\begin{proof}
Note that $w_0\in C(\bar U)$, $G\in C(\Gamma\times[0,T])$.
By Theorem~\ref{maximumprinciple},  $w(\x,t)$ is bounded on $\bar U_T$.  Let $\varep$ be an arbitrary positive number. 
There is $\tilde r_0>0$ such that for $|\x|>\tilde r_0$ we have 
\beq
 |w_0(\x)|<\varep\quad   \text{ and }\quad \sup_{0\le t\le T}  |\fo(\x,t)| < \varep.  
\eeq

Let $\tilde w =\pm w -\varep(t+1)$ then $\tilde w$ is bounded on $\bar U_T$ and $\mathcal L \tilde w <0 $ on $\mscrO_{\tilde r_0}\times(0,T]$, and $\tilde w(\x,0)\le 0$ on $\bar \mscrO_{\tilde r_0}$.  Applying Proposition~\ref{small0} to $\tilde w$ with $r_0$ being replaced by $\tilde r_0$ gives 
\beqs
\limsup_{r\to\infty} (\sup_{\Sph_r\times[0,T]}\tilde w(\x,t))\le 0.
\eeqs 
This implies  
\beqs
\limsup_{r\to\infty} (\sup_{\Sph_r\times[0,T]}[\pm w(\x,t)])\le \varep(T+1).
\eeqs 
Therefore,
\beqs
 \limsup_{r\to\infty} (\sup_{\Sph_r\times[0,T]} |w(\x,t)| ) \le \varep (T+1).
\eeqs
Letting $\varep\to 0$ we obtain \eqref{assym3}.
\end{proof}

We now consider problem \eqref{nonhom} for all $t>0$ under condition \eqref{w0lim}.
Although it is not known whether $\lim_{t\to\infty}w(\x,t)$ exists for each $\x$,  we prove in the corollary below that such limit is zero along some curve $\x(t)$ which goes to infinity as $t\to\infty$.

 
 \begin{corollary}\label{Reg1}  
Let $w(\x,t)\in C_{\x,t}^{2,1}(D)\cap C(\bar D)$ be a bounded solution of \eqref{nonhom} on $D$
with $\fo\in C(\bar D)$.
Assume $w_0\in C(\bar U)$ satisfies \eqref{w0lim}, $G\in C(\Gamma\times[0,\infty))$ is bounded, and \eqref{focond} holds for each $T>0$.
Then there exists an increasing, continuous function $r(t)>0$ satisfying $\lim_{t\to\infty} r(t) =\infty$ such that
  \beq\label{limit0} 
    \lim_{ t\to\infty}\Big( \sup_{\x\in \bar \mscrO_{r(t)}} |w(\x,t)|\Big)=0.
 \eeq
\end{corollary}
 \begin{proof} 
By Theorem \ref{small1},  there exists a strictly increasing sequence $\{r_k\}_{k=1}^\infty$ of positive numbers such that $\lim_{k\to\infty}r_k=\infty$ and
\beq\label{reg}
\sup_{\{\x:|\x|\ge r_k\}\times [0,k]}| w(\x,t) |<\frac 1 k .
\eeq
Let $r(t)$ be the piecewise linear function passing through the points $(k,r_{k+1})$ then $r(t)$ is increasing and  $r(t)\to\infty$ as $t\to\infty$.
By \eqref{reg}, for each $k$ we have
\beqs
\sup\{ |w(\x,t)|: k\le t\le k+1, |\x|\ge r(t)\}\le \sup_{\{\x:|\x|\ge r_{k+1}\}\times [0,k+1]} |w(\x,t)|<\frac 1 {k+1}.
\eeqs
 Taking $k\to\infty$ we obtain \eqref{limit0}. 
 \end{proof}

\medskip
We now return to the IBVP \eqref{sigEq} for $\sigma$. We will use the transformation $\sigma = we^{\Lambda}$. To compare $\sigma$ and $w$, we need to estimate $\Lambda(\x)$. Recall from \eqref{Lambdform} that
\begin{align*}
\Lambda (\x) 
&=\int_{r_0}^{|\x|} \tilde F(r)dr,\text{ where } \tilde F(r)= F'_2(\hat S(r)) g_2(\frac{|c_2|}{r^{n-1}} )\frac{c_2}{r^{n-1}} - F'_1(\hat S(r)) g_1(\frac{|c_1|}{r^{n-1}} )\frac{c_1}{r^{n-1}}.
\end{align*}
For $R$ sufficiently large and $r\ge R$, we have $ | \tilde F(r)|\le C r^{1-n}$. 
Then we have in the case $n\ge 3$ that  $| \tilde F(r)|\le C r^{-2} $, hence $|\Lambda(\x)|\le C_6$ for all $|\x|\ge r_0$, and 
\beq\label{bdness}
0<C_7^{-1}\le e^{\Lambda(\x)}\le C_7 \quad\forall |\x|\ge r_0.
\eeq


\begin{theorem}\label{uniqS3}
Let $n\ge 3$. Assume \condEa\ and 
\beq\label{Delele}
\deltaele \eqdef \max\{\sup_{U }| \sigma_0(\x)|,\sup_{\Gamma\times[0,\infty)}|g(\x,t)|\} <\infty, 
\eeq    
\beq\label{tiny1}
 \deltaei \eqdef \sup_{D} | \nabla\cdot(\Am(\x) \mathbf c(\x,t))| <\infty.
\eeq
Then,

{\rm (i)} There exists a solution $\sigma(\x,t)\in C^{2,1}_{\x,t}(D)\cap C(\bar  D)$ of problem \eqref{sigEq}. This solution is unique in class of solutions $\sigma(\x,t)$ that satisfy
\beq
\sup_{ U\times [0,T]} |\sigma(\x,t)|<\infty\quad \text{for any } T>0.
\eeq

{\rm (ii)} There is $C>0$ such that for $(\x,t)\in D$,
\beq\label{supbound}
 |\sigma(\x,t)|\le C\big[ \deltaele + \deltaei (t+1)\big] .
\eeq

{\rm (iii)} In addition,  if
\beq\label{sigzerocond}
 \lim_{|\x|\to\infty}  \sigma_0(\x) =0
\quad \text{and}\quad 
\lim_{|\x|\to\infty} \sup_{0\le t\le T} |\nabla\cdot(\Am(\x)\mathbf c(\x,t))| =0 \text{ for each }T>0, 
\eeq    
then
\beq\label{limzero}
\lim_{r\to\infty} \Big(\sup_{\Sph_r\times [0,T]}| \sigma(\x,t) |\Big)=0 \quad \text{for any } T>0,
\eeq
and furthermore, there is a continuous, increasing function $r(t)>0$ with $\lim_{t\to\infty}r(t)=\infty$ such that
\beq\label{sigmaCurve}
 \lim_{t\to\infty} \Big(\sup_{\x\in\bar\mscrO_{r(t)}} |\sigma(\x,t)| \Big) =0.
\eeq
\end{theorem}
\begin{proof}
  Let  $w_0(\x)=\sigma_0(\x) e^{-\Lambda(\x)}$, $G(\x,t)=g(\x,t) e^{-\Lambda(\x)}$ and $\fo(x,t)=e^{-\Lambda(\x)}|\nabla\cdot(\Am(\x) \mathbf c(\x,t))|$.
Thanks to \eqref{bdness} and \eqref{Delele}, we have 
  \beqs
   \max\{\sup_{U }| w_0(\x)|,\sup_{\Gamma\times [0,\infty)}|w(\x,t|\}  \le  C \deltaele,
  \eeqs
\beqs
\sup_{D}  |\fo| \le C\deltaei. 
\eeqs
Then statements in (i), (ii) and (iii) follow directly from Theorems \ref{Euniq} and \ref{small1}, and Corollary \ref{Reg1} for problem \eqref{nonhom}, the relation $\sigma(\x,t)=w(\x,t) e^{\Lambda(\x)}$ and the boundedness of $e^{\Lambda(\x) }$ in \eqref{bdness}. We omit the details.
\end{proof}

As a consequence of \eqref{sigmaCurve},  for any continuous curve $\x(t)$ with $|\x(t)|\ge r(t)$, one has
\beq
\lim_{t\to\infty} \sigma(\x(t),t)=0.
\eeq

The case $n=2$ is treated next with some restriction on the steady state.

\begin{theorem}\label{sta1}
Let $n=2$ and  $\hat S(r)$ be a solution of \eqref{IVP} with $c_1,c_2<0$.
Assume \condEa\  and 
\beq\label{tilM2}
\deltatwe \eqdef \max\{\sup_{ U }e^{-\Lambda(\x)} | \sigma_0(\x)|, \sup_{\Gamma\times [0,\infty)}|g(\x,t)|\} <\infty,
\eeq    
\beq\label{n2cond}
\deltani \eqdef \sup_{D} e^{-\Lambda {(\x)}}| \nabla\cdot(\Am(\x) \mathbf c(\x,t))| <\infty.
\eeq

Then the following statements hold true.

{\rm (i)} There exists a solution $\sigma(\x,t)\in C^{2,1}_{\x,t}(D)\cap C(\bar D)$ of problem \eqref{sigEq}. 
This solution is unique in class of solutions $\sigma(\x,t)$ that satisfy
\beq
\sup_{U\times [0,T]} e^{-\Lambda(\x)} |\sigma(\x,t)|<\infty\quad \text{for any } T>0.
\eeq

{\rm (ii)} There is $C>0$ such that for $(\x,t)\in D$,
 \beqs
|\sigma(\x,t)|\le C\big[\deltatwe + \deltani (t+1)\big].
 \eeqs
 
{\rm (iii) } Statement {\rm (iii)} of Theorem \ref{uniqS3} holds true if condition \eqref{sigzerocond} is replaced by 
\beq
 \lim_{|\x|\to\infty}  e^{-\Lambda(\x)}\sigma_0(\x) =0 \quad\text{and}\quad \lim_{|\x|\to\infty} \sup_{0\le t\le T} e^{-\Lambda(\x)} |\nabla\cdot(\Am(\x)\mathbf c(\x,t))| =0 \text{ for each }T>0. 
\eeq    
\end{theorem}
\begin{proof} 
According to Theorem \ref{special2}, 
 $\lim_{r\to\infty}\hat S(r)=s^* \in (0,1),$ where $s^*$ is defined in \eqref{sstar2}.
The proof consists of two steps.

\textit{Step 1.} We show that statements (i)--(iii) hold true under the following condition
\beq\label{negCond}
F_2'(s^*) a_2^0 c_2 - F_1'(s^*)a_1^0 c_1 <0. 
\eeq

Let $c_4=-(F_2'(s^*) a_2^0 c_2 - F_1'(s^*)a_1^0 c_1)>0$. We have for any $R>r_0$ and $|\x|>R$ that
\beqs
\Lambda(\x)=\int_{r_0}^R  \tilde F(r) dr  +\int_{R}^{|\x|}  \tilde F(r)dr  =I_1(R)+I_2(R).
\eeqs 
For sufficiently large $R_0>r_0$, we have for $|\x|>R_0$ that 
 \beqs
  I_2(R_0)\le \frac 1 2\int_{R_0}^{|\x|}  \big( F'_2(\hat S(r)) a_2^0c_2 - F'_1(\hat S(r)) a_1^0c_1 \big) r^{-1} dr 
\le -\frac 14\int_{R_0}^{|\x|} c_4 r^{-1} d\xi\le 0.
  \eeqs
Obviously, $I_1(R_0)$ is finite. This gives $e^{\Lambda(\x)}\le C_{8}<\infty$ for all $|\x|\ge r_0$. Thus, 
  \beq\label{sigw}
 |\sigma| \le C_{9} |w| \quad \text{with constant }C_{9}>0. 
  \eeq
  Setting $w(\x,t)=\sigma(\x,t) e^{-\Lambda(\x)}$, we have $\mathcal L w =\fo$, where $\fo$ is as in Theorem \ref{uniqS3}.
Then (i)--(iii) easily  follow Theorems~\ref{Euniq}, \ref{small1}, Corollary \ref{Reg1} and relation \eqref{sigw}.
          
\textit{Step 2.} Now, it suffices to show that condition \eqref{negCond} is satisfied  with $c_1,c_2<0$.
On the one hand, we have from \eqref{sstar2} that 
\beqs
\frac {a_1^0c_1}{a_2^0c_2}= f(s^*) = \frac{f_1}{f_2}(s^*) =\frac {F_2(s^*)}{F_1(s^*)}.
\eeqs
Then $a_1^0c_1F_1(s^*)= a_2^0c_2F_2(s^*)\eqdef\mathcal A \neq 0$. 
On the other hand,  
\begin{align*}
F_2'(s^*) a_2^0 c_2 - F_1'(s^*)a_1^0 c_1 =  \mathcal A\Big[ \frac {F_2'(s^*)}{F_2(s^*)} - \frac{F_1'(s^*)}{F_1(s^*)}\Big]
=\mathcal A   \frac{F_1(s^*)}{F_2(s^*)}\Big(\frac {F_2}{F_1}\Big)'(s^*).
\end{align*}
The assumptions on $f_1$ and $f_2$ provide $(F_2/F_1)'(s^*) =( f_1/f_2)'(s^*)>0$ and $F_1(s^*),F_2(s^*)>0$. 
Since $c_1,c_2<0$, we have $\mathcal A<0$  and, hence, $F_2'(s^*) a_2^0 c_2 - F_1'(s^*)a_1^0 c_1<0$.
The proof is complete.
\end{proof}

\begin{remark}
 Similar to Theorem \ref{mainv}, we can use Bernstein's technique to estimate $\V_1(\x,t)$ and $\V_2(\x,t)$ uniformly in $\x\in U'\Subset U$. We do not provide details here.
\end{remark}

%

\myclearpage

\appendix

\section{}

We give proof to the statements on the range of $s_\infty$ in Example~\ref{exam}.   
Recall that $s_\infty\in [0,1]$.

In the case $\Delta=0$ of A and B, $h(r)\equiv s^*$ is the equilibrium and the conclusions are clear.
Also, for C and D, $S(r)$ is monotone and the statements easily follow. 
We focus on the remaining cases.

A.  $c_1,c_2>0$. Note that $F(r,S)>0$ iff $S>h(r)$, hence $S'(r)>0$ iff $S(r)>h(r)$.
 \begin{itemize}
 \item $\Delta<0$. Then $h(r)$ increases and $h(r)<s^*$ for all $r$.
Consider $s_0>s^*$. Then $S(r)>s^*>h(r)$ for all $r$. It follows that $S(r)$ is strictly increasing which implies $s_\infty>s_0$.
Now, consider $s_0<h(r_0)$. Then $S(r)<h(r)$ for all $r$, thus $S(r)$ is strictly decreasing and, therefore, $s_\infty<s_0$.

\item $\Delta>0$. In this case, $h(r)$ is decreasing, and $h(r)>s^*$ for all $r$. Then the arguments are the same as in the case $\Delta<0$.

\end{itemize}

B. $c_1,c_2<0$. Observe that $F(r,S)>0$ iff $S<h(r)$, hence $S'(r)>0$ iff $S(r)<h(r)$. 
\begin{itemize}
 \item $\Delta<0$. Then $h(r)$ is increasing and $h(r)<s^*$ for all $r$.

We prove (iii) first when $s_0<h(r_0)$. Exactly the same as Claim 2 in the proof of Theorem \ref{genS}, we have $S(r)\le h(r)<s^*$ for all $r$. 
  Thus $S(r)$ is increasing on $[r_0,\infty)$. Hence $s_\infty\in [s_0,s^*]$. Since $S(r)$ is strictly increasing for $r$ near $r_0$, we have $s_\infty>s_0$.

We prove (ii). Consider the subcase $h(r_0)< s_0 \le s^*$. Then there exists $r_1>r_0$ such that $S(r)>h(r)$ for $r<r_1$ and $S(r_1)=h(r_1)$.
Similar arguments to (iii), we have $S(r_1)\le S(r)\le h(r)$ for all $r<r_1$. Hence $s_\infty\le s^*$ and $s_\infty\ge h(r_1)>h(r_0)$.

In the particular case $s_0=h(r_0)$, one can show that $h(r_0)\le S(r)\le h(r)$ for all $r>r_0$.
If $S(r)\equiv h(r)$ then $s_\infty=s^*$. Otherwise, there is $r_1>r_0$ and such that $h(r_0)\le S(r_1)<h(r_1)$.
Similar to (iii) with $r_0$ playing the role of $r_1$, we have $s_\infty \in ( S(r_1),s^*]$. Hence $s_0\in (h(r_0),s^*]$.

Finally, we prove (i) when $s_0>s^*$. Clearly, $S(r)<s_0$ for all $r>r_0$. 
If $s_0>S(r)> s^*$ for all $r>r_0$ then we have $S(r)$ strictly deceasing and $s_\infty\in [s^*,s_0)$.
Otherwise, there is $r_1$ such that $S(r_1)=s^*$. Then using (ii) we obtain $s_\infty \in (h(r_0),s_*]$.

\item $\Delta>0$. Then $h(r)$ is decreasing, and $h(r)>s^*$ for all $r$. The proof is similar to the case $\Delta<0$.
 \end{itemize}

\myclearpage
 \def\cprime{$'$}

 \bibliographystyle{abbrv}

\end{document}